\documentclass[11pt]{article}

\usepackage{graphics,enumitem,epsfig,textcomp}
\usepackage{amsfonts,amsmath,amssymb,amsthm}
\usepackage{euscript,color,mathrsfs}

\usepackage{booktabs}

\usepackage{url}
\usepackage{hyperref}
\hypersetup{
    colorlinks=true,
    linkcolor=blue,
    filecolor=blue,      
    urlcolor=blue,
    citecolor=blue,
    }

\usepackage[mathlines]{lineno}
\usepackage[margin=1in]{geometry}

\numberwithin{equation}{section}
\newtheorem{lemma}{Lemma}%[section]
\newtheorem{corollary}{Corollary}%[section]
\newtheorem{remark}{Remark}%[section]

\def\d{\,\mathrm{d}}
\def\N{\mathbb{N}}
\def\Z{\mathbb{Z}}
\def\R{\mathbb{R}}
\def\C{\mathbb{C}}

\def\({\begin{eqnarray}}
\def\){\end{eqnarray}}
\def\[{\begin{eqnarray*}}
\def\]{\end{eqnarray*}}
\def\part#1#2{\frac{\partial #1}{\partial #2}}

\def\grad{\nabla}
\def\Norm#1{\left\| #1 \right\|}
\def\bar{\overline}

\def\tot#1#2{\frac{\d #1}{\d #2}} 

\def\laplace{\Delta}
\def\d{\,\mathrm{d}}
\def\N{\mathbb{N}}
\def\R{\mathbb{R}}

\def\epsilon{\varepsilon}

\def\d{\mathrm{d}}

\def\AK{A}  %A\rule{0pt}{9pt}^{(K)}}

\def\conv{\!\ast\!}

\def\x{x}
\def\y{y}
\def\z{z}

\def\yy{\mathbf{y}}

\def\ee{\mathbf{e}}

\def\Re{\operatorname{Re}}

\usepackage{authblk}

\title{Memory Stabilizes Spontaneous Particle Aggregation:\\ A Linearized Vlasov--Fokker--Planck Analysis}
\date{}

\author{
     Jan Haskovec\\ Mathematical and Computer Sciences
            and Engineering Division,\\
         King Abdullah University of Science and Technology,\\
         Thuwal 23955-6900, Kingdom of Saudi Arabia\\
         {\it jan.haskovec@kaust.edu.sa}
     }

\begin{document}

\maketitle

\begin{abstract} \noindent
We perform a linearized stability analysis of the Vlasov--Fokker--Planck equation obtained as the mean-field description of a stochastic spontaneous aggregation model with memory. Memory is represented by a chain of $K$ internal variables.
% through which individuals retain information about previously encountered population densities.
We characterize the spatially homogeneous equilibria and derive a scalar dispersion relation for spatially inhomogeneous perturbations. When all relaxation rates are equal, we show that every Fourier mode that is unstable in the memoryless model possesses a unique critical relaxation parameter: sufficiently long memory stabilizes the mode, whereas it remains unstable for short memory. Although memory may destabilize individual modes associated with negative Fourier coefficients of the sensing kernel, we prove that, for radially symmetric distance-decreasing kernels, it cannot destabilize a homogeneous equilibrium that is stable in the memoryless model. Thus, memory cannot create an overall instability of an otherwise stable homogeneous state, although it may change the set of the unstable modes.
We present a numerical example for the normalized top-hat kernel, demonstrating that increasing the effective memory length
successively stabilizes the Fourier modes, with higher frequencies being stabilized before lower ones.
This is consistent with the coarsening effect observed in recent particle simulations.
Finally, under suitable nondegeneracy and regularity assumptions, we use the Crandall--Rabinowitz theorem to obtain branches of spatially inhomogeneous stationary solutions bifurcating from the homogeneous state.
\end{abstract}
\vskip 5mm

\vskip 5mm

%35Q84: Fokker-Planck equations
%35B35: Stability in context of PDEs
%35H10: Hypoelliptic equations
%35B36: Pattern formations in context of PDEs

%%%%%%%%%%%%%%%%%%%%%%%%%%%%%%%%
\section{Introduction}\label{sec:Intro}
Collective spatial organization is a fundamental phenomenon in biological systems, observed across a wide range of scales, from bacterial and amoeboid motion to the behaviour of social insects, fish schools and bird flocks; see, e.g., \cite{BergBrown:1972, Erban:2004, Sumpter:2006, Sumpter:2010, Vicsek:2012}. A particularly simple mechanism for the emergence of aggregation is density-dependent random motility: individuals modulate the amplitude of their random motion according to the locally perceived population density via a nonlinear response function $G>0$.
The population density in each agent's vicinity is sensed through a radially symmetric distance-decreasing sensing kernel $W$.
If the motility decreases in regions of high density, then the resulting continuum model contains an effective aggregating drift, competing with diffusion. This mechanism is captured by the spontaneous aggregation model introduced in \cite{BHW:2012}. In many biological situations, however, individuals do not respond only to the instantaneous state of their environment, but also to previously encountered stimuli. It is therefore natural to incorporate memory, or delayed response, into models of biological collective phenomena, see, e.g., \cite{Boschi:2021, Couzin:2002, Erban:2005, Fagan:2023, Falcon:2023, Kim:2024, Sun:2014}.

The stochastic particle model with memory introduced in \cite{EH:2026} extends the direct aggregation model of \cite{BHW:2012} by equipping each agent with a chain of $K$ internal variables. These variables store information about previously perceived densities and determine the amplitude of the stochastic motion. In the case when all internal variables have the same relaxation coefficient $\alpha>0$, the corresponding memory kernel can be evaluated explicitly and has effective memory length $K/\alpha$. Systematic numerical simulations carried out in \cite{EH:2026} indicate that memory has a significant influence on pattern formation: moderate-term memory promotes coarsening, producing fewer but larger clusters, whereas long-term memory inhibits aggregation. The goal of the present paper is to give an analytical explanation of this behaviour by studying the linearized stability of spatially homogeneous steady states of the corresponding Vlasov--Fokker--Planck equation.

We first derive the general spatially homogeneous equilibria of the Vlasov--Fokker--Planck equation. In contrast to the singular relaxation limit considered in \cite{EH:2026}, where the invariant measure is concentrated on a lower-dimensional manifold, the finite relaxation model admits a full Gaussian invariant measure in the internal variables. We then linearize the equation around a spatially homogeneous equilibrium and decompose the perturbation into spatial Fourier modes. For each nonzero mode we derive a scalar dispersion relation. In the case of equal relaxation rates, 
the dispersion relation yields an explicit instability criterion for every mode with a positive Fourier coefficient of the sensing kernel.

Our analysis shows that memory may only suppress the aggregation instabilities inherited from the memoryless model and cannot destabilize an otherwise stable homogeneous equilibrium. More precisely, for each Fourier mode that is unstable in the memoryless model, there exists a unique critical value of the relaxation parameter $\alpha>0$ such that the mode is unstable for sufficiently large $\alpha$, corresponding to short memory, and stable for sufficiently small $\alpha$, corresponding to long memory. 
However, as we shall demonstrate with an example, memory can destabilize a particular Fourier mode that was stable without memory.
This may happen for modes associated with negative Fourier coefficients of the interaction kernel.
Nevertheless, our analysis shows that this can occur only if the homogeneous state is already unstable in the memoryless model.
The main conclusion of our analysis can therefore be formulated as follows:
\begin{quote}
\bfseries \itshape
For radially symmetric distance-decreasing sensing kernels $W$,
memory cannot create an instability of a stable homogeneous state of the memoryless model.
However, it may change the set and the nature of the unstable modes, thus modifying the aggregation patterns
produced by the model.
\end{quote}
We also show that for the normalized top-hat kernel considered in the numerical examples,
the modification of the aggregation patterns is such that higher Fourier modes are stabilized before lower ones
as the effective memory length increases. This is consistent with the coarsening effect observed numerically in \cite{EH:2026}.

We also discuss the existence of spatially inhomogeneous stationary solutions in the spatially one-dimensional setting.
Under suitable nondegeneracy and smoothness assumptions, the Crandall--Rabinowitz bifurcation theorem yields a branch of inhomogeneous steady states that bifurcates from the homogeneous branch whenever the linearized stationary operator has a simple zero eigenvalue.
In particular, for the normalized top-hat interaction kernel, we show that if the first Fourier mode is critical, then the simplicity condition is automatically satisfied.

Finally, we present a numerical evaluation of the roots of the dispersion relation for the response function $G(s)=e^{-s}$ and the normalized top-hat sensing kernel, which is the setting used both in \cite{BHW:2012} and \cite{EH:2026}.
For the first few Fourier modes and for $K=1,2,3$, we compute the critical values of $\alpha$ and plot the corresponding real roots of the dispersion relation. These computations illustrate the analytical conclusions: for fixed $K$, increasing the effective memory length by decreasing $\alpha$ reduces the number of unstable Fourier modes and therefore leads to less oscillatory aggregation patterns, or inhibits the aggregation completely once $\alpha$ becomes small enough.

The paper is organized as follows. In Section~\ref{sec:TheModel} we recall the stochastic spontaneous aggregation model with memory and its Vlasov--Fokker--Planck mean-field description. Section~\ref{sec:stability} is devoted to the derivation and analysis of the dispersion relation for the linearized equation. In Section~\ref{sec:Rabinowitz} we prove, employing the Crandall--Rabinowitz theorem, the bifurcation of spatially inhomogeneous stationary solutions from the homogeneous branch. Section~\ref{sec:visual} contains the numerical evaluation of the dispersion relation for a generic one-dimensional setting. In Section \ref{sec:destab} we give an example of destabilization of a particular Fourier mode through the introduction of memory. Finally, the Appendix, Section \ref{sec:App}, establishes useful facts about the spectrum of the free kinetic operator, and, moreover, provides explicit dispersion relations for $K=1,2,3$.

%%%%%%%%%%%%%%%%%%%%%%%%%%%%%%%%%%
\section{The model}\label{sec:TheModel}
The spontaneous particle aggregation model with memory introduced in~\cite{EH:2026} consists of a group of $N \ge 2$ biological agents,
characterized by their positions ${\x}_i(t)\in\Omega$, $i\in [N]$, where $\Omega:=(0,1)^d$, $d \geq 1$, is the $d$-dimensional torus.
We use the notation $[N] := \{1,2,\ldots,N \}$ throughout the paper.
Every individual senses the average density of its close neighbors, given by
\begin{equation}   
\label{rho_i}
\vartheta_i(t) = \frac{1}{N-1} \sum_{\substack{j\in [N] \\ j \neq i}} W({\x}_i(t)-{\x}_j(t)) \qquad \mbox{for} \quad i \in [N].
\end{equation}
The sensing kernel $W=W({\x})$ is assumed to be nonnegative, bounded, radially symmetric and nonincreasing with respect to $|\x|$.
We regard $W$ as a compactly supported function on $\mathbb R^d$, periodically extended to the torus $\Omega$, and normalized by
\begin{equation}  
\label{eq:Wnorm}
\int_{\Omega} W({\x}) \, \d {\x} = 1 \,. 
\end{equation}
A prototypical example is the normalized characteristic function of the ball $B_R(0)$,
corresponding to the sampling radius $0 < R < 1/2$.
The average density $\vartheta_i$ is then simply the fraction
of individuals located within the distance $R$ from the $i$-th individual.

The agents are equipped with memory, realized as the chain of internal variables
${\y}_{i}^{k}={\y}_{i}^{k}(t)\in\R^d$ for $i \in [N]$ and $k\in [K]$,
where $K\geq 1$ is the number of memory `layers'.
The dynamics of the group is then governed by the following system of SDEs,
\begin{equation}   
\label{eq:SDE0}
\begin{aligned}
\tot{{\x}_i(t)}{t} &\,=\, {\y}_{i}^1(t)\,, \\
\varepsilon_k \tot{{\y}_{i}^k(t)}{t} &\,=\, - \alpha_k \,{\y}_{i}^k(t) \,+\, {\y}_{i}^{k+1}(t)\,, \qquad\quad \mbox{for} \;\; k \in [K-1]\,, \\
\varepsilon_K \,\d {\y}_{i}^K(t) &\,=\,  - \alpha_K \,{\y}_{i}^K(t) \, \d t \,+\, G(\vartheta_i(t)) \,\d{B}_i^t\,,
\end{aligned}
\end{equation}
where ${B}_i^t$ are independent $d$-dimensional Brownian motions.
The positive constants $\alpha_k > 0$ are the relaxation coefficients,
and the parameters $\varepsilon_k>0$ define the relaxation time scales.
The {response function} $G:\R^+ \to\R^+$ is assumed to be bounded, positive and decreasing.
The monotonicity of $G$ is implied by the modeling assumption that the individuals respond to 
higher perceived population densities $\vartheta_i = \vartheta_i(t)$ in their vicinity by reducing the amplitude of 
their stochastic movement.

For a slight simplification of the analysis and notation, we shall fix unit relaxation time scales in the sequel,
\begin{equation} \label{epsare1}
   \varepsilon_k:=1 \qquad\quad \mbox{for all} \quad k\in[K].
\end{equation}
%Moreover, we restrict ourselves to the spatially one-dimensional setting, $d=1$.
%This restriction, in fact, is merely to ease the notation, and poses no loss of generality.
%The multidimensional case is analysed, in essence, simply by treating \eqref{SDE1}
%for each spatial component, coupled only through the terms $\vartheta_i = \vartheta_i(t)$, $i\in[K]$.
For each $i\in[N]$ we denote $\mathsf{y}_i := \left( y^1_i,\dots,y^K_i \right)^T \in\mathbb R^{dK}$,
and $\mathsf{y} := \left( \mathsf{y}_1,\dots, \mathsf{y}_N \right) \in\mathbb R^{dKN}$.
The SDE system \eqref{eq:SDE0} with \eqref{epsare1} can then be written, for each
spatial component $\mathsf{y}_i^\ell = \mathsf{y}_i^\ell(t)\in \R^{K}$, $\ell\in [d]$,
in matrix form as
\begin{equation}
\label{eq:SDE1}
   \d{\mathsf{y}_i^\ell(t)} \,=\, A \mathsf{y}_i^\ell(t) \, \d t \,+  \, G(\vartheta_i(t)) \, \ee_K \, \d B_{i,\ell}^t,
\end{equation}
%where $B_i^t$, $i\in[N]$, are independent $d$-dimensional Brownian motions, and
where $\ee_K:=(0,\dots,0,1)^T\in\mathbb R^K$ denotes the $K$-th standard basis vector and
the constant matrix $A\in \R^{K\times K}$ %and the vectors ${\boldsymbol{\beta}}_i = {\boldsymbol{\beta}}^{(K)}_i(t) \in \R^{K}$
is given by
\begin{equation}
\label{def:A}
  A := \begin{pmatrix}
      - \alpha_1 & 1 & 0 & 0 & \ldots & 0 \\
      0 & -\alpha_2 & 1 & 0 & \ldots & 0 \\
      \vdots & & \ddots & & 1 & 0 \\
      0 & &  & & -\alpha_{K-1}& 1 \\
      0 & & \ldots & & 0 & -\alpha_K
   \end{pmatrix}.
%\qquad\quad\mbox{and}\qquad\quad {\boldsymbol{\beta}}_i(t) := \begin{bmatrix} 0 \\ 0 \\ \vdots \\ 0 \\ G(\vartheta_i(t)) \end{bmatrix}.
\end{equation}

As noted in~\cite[Section 3]{EH:2026}, model~\eqref{eq:SDE1} can also be formally rewritten in a non-Markovian form,
eliminating the internal variables $\mathsf{y}=\mathsf{y}(t)$.
Indeed, using the variation-of-constants formula~\cite[Section 3.3]{Mao:2007} for~\eqref{eq:SDE1}
and choosing the zero initial datum $\mathsf{y}(t=0)=0$, we arrive at
\(  \label{eq:nonMarkovian}
   \tot{{\x}_i(t)}{t} = \int_0^t \kappa(t-s) \,G(\vartheta_i(s)) \,\d B_{i}^s
   \qquad \mbox{with} \qquad
   \kappa(t) := \left( e^{tA} \right)_{1,K},
\)
i.e., $\kappa(t)$ is the $(1,K)$-th element of the matrix exponential $e^{tA}$.
This can be evaluated explicitly under the assumption
\begin{equation}
   \alpha_k:=\alpha>0 \qquad\quad \mbox{for all} \quad k\in[K]\,.
\label{eq:samealpha}
\end{equation}
Then $A = -\alpha I + U$, with $U$ the upper-diagonal matrix $U_{k,k+1}=1$ for $k\in[K-1]$.
Since $U$ commutes with the identity matrix, we have
\(  \label{eq:kappa}
   \kappa(t) = e^{-\alpha t} \left( e^{tU} \right)_{1,K} = \frac{e^{-\alpha t} \, t^{K-1}}{(K-1)!}\,.
\)

As noted in \cite[Section 3]{EH:2026}, the non-Markovian formulation \eqref{eq:nonMarkovian} can be used
to establish a biologically meaningful notion of effective memory duration,
i.e., the mean length of the past time interval that influences the current motion of an individual.
This quantity corresponds to the mean of the normalized version of the kernel $\kappa=\kappa(t)$,
viewed as a weighting function in \eqref{eq:nonMarkovian}. %assigning importance to past sensory inputs.
In particular, for $\kappa=\kappa(t)$ given by \eqref{eq:kappa}, we readily calculate
\( \label{eq:MemoryLength}
   \left( \int_0^{+\infty} \kappa(t) \, \d t \right)^{\!\!-1} \left( \int_0^{+\infty} t \, \kappa(t) \, \d t \right) = \frac{K}{\alpha}.
\)
We observe that, for a fixed $K\in\N$, the effective length of the memory is
inversely proportional to the value of $\alpha>0$ in \eqref{eq:samealpha}.
The main goal of this paper is to study how the memory length
influences the pattern formation (aggregation) properties of the system \eqref{eq:SDE1}.
We shall proceed by a linearized stability analysis of the spatially homogeneous steady states
of the corresponding Vlasov--Fokker--Planck equation.

%%%%%%%%%%%%%%%%%%%%%%%%%%%%%%%%%%%%%%%%
\subsection{The Vlasov--Fokker--Planck equation}
In~\cite{EH:2026}, the mean-field description of \eqref{rho_i}--\eqref{eq:SDE0} as $N\to+\infty$ has been formally derived,
resulting in the kinetic Fokker--Planck equation
\(  \label{eq:FP}
  \begin{aligned} 
     \partial_t f  + {\y}^1\cdot \grad_{\x} f 
     \,+  \sum_{k=1}^{K-1}  \varepsilon^{-1}_k \grad_{{\y}^k}\cdot \left( ({\y}^{k+1} - \alpha_k {\y}^k) f \right)  \\
       = \varepsilon_K^{-1} \grad_{{\y}^K}\cdot \left( \alpha_K {\y}^K f + \frac{G(W\conv\rho[f])^2}{2} \grad_{{\y}^K}  f \right),
\end{aligned}
\)
for the particle phase-space density $f = f(t, {\x}, \yy)$, where we introduced the notation
\[
   \yy := \left( y^1,\dots,y^K \right) \in\mathbb R^{dK},
\]
and for each spatial component $\ell\in [d]$ we denote
$\yy_\ell := \left( y^1_\ell,\dots,y^K_\ell \right)^T \in\mathbb R^K$.
The particle number density $\rho[f]$ is the zeroth $\yy$-moment of $f$,
\begin{equation}   \label{def:rho}
   \rho[f] (t,{\x}) := \int_{\R^{dK}}  f(t, {\x}, {\yy})   \, \d {\yy},
\end{equation}
and the convolution $W\conv \varrho$ over the torus $\Omega$ is defined as
\(  \label{def:convolution}
   W\conv \rho[f](t, {\x}) := \int_{\R^d} W({\x}-{\z}) \rho_\mathrm{per}[f]({\z}) \, \d {\z},
\)
where $\rho_\mathrm{per}[f]$ is the periodic extension of $\rho[f]$ to $\R^d$.

To avoid unnecessary technicalities, we again 
fix unit relaxation time scales \eqref{epsare1} in the sequel.
We also drop the index $\ell$ in $\yy_\ell$ and interpret the below formulas as componentwise in $\ell\in[d]$.
Then \eqref{eq:FP} can be written in the compact notation
\begin{equation}
   \label{eq:FPcompact}
   \partial_t f + \y^1\cdot\nabla_\x f = \mathcal L[\rho[f]] f,
\end{equation}
with $\varrho:=\rho[f]$ given by \eqref{def:rho} and
\(  \label{def:L}
   \mathcal L[\varrho] f := -\grad_\yy\cdot\left( A\yy f \right) + \frac{G(W\ast\varrho)^2}{2}\laplace_{y^K} f,
\)
with the matrix $\AK \in \R^{K\times K}$ given by \eqref{def:A}.

%%%%%%%%%%%%%%%%%%%%%%%%%%%%%
\subsection{Properties of the model with no memory} \label{subsec:noMemory}
Setting formally $K:=0$ in \eqref{eq:SDE1}, we obtain the spontaneous particle aggregation model without memory,
\(  \label{eq:noMemory}
   \d {\x}_{i}(t) =  G(\vartheta_i(t)) \,\d {B}_i^t\,,
\)
with $\vartheta_i = \vartheta_i(t)$ given by \eqref{rho_i}.
This model has been proposed and studied in~\cite{BHW:2012} under the name `direct aggregation'.
The formal mean-field limit as $N\to\infty$ has been derived for the particle number density $\varrho = \varrho(t,{\x})$,
\begin{equation}
   \label{eq:FPnoMemory}
     \part{\varrho}{t} = \frac12 \, \laplace\left( G(W\ast\varrho)^2 \varrho \right),
\end{equation}
with the convolution again given by \eqref{def:convolution}.
To obtain an insight into the dynamics of the model, it is instructive to expand the derivative in the right-hand side of~\eqref{eq:FPnoMemory},
\begin{equation*}
   \part{\varrho}{t} =  \frac12 \, \grad\cdot\Big( \varrho \, \grad G(W\ast\varrho)^2 + G(W\ast\varrho)^2 \, \grad \varrho \Big).
\end{equation*}
We observe that the convection term $\grad\cdot\left( \varrho \, \grad G(W\ast\varrho)^2 \right)$ %(the first term on the right-hand side)
induces the eventual formation of aggregates, competing against the smoothing action of the purely diffusive term $\grad\cdot \left(G(W\ast\varrho)^2 \, \grad \varrho \right)$.
Stationary patterns, emerging for suitable choices of the model parameters, result as equilibria balancing these two mechanisms.

A linearized stability analysis of a given constant steady state $\varrho \equiv \varrho_0 > 0$
has been carried out in~\cite{BHW:2012}. It leads to the system of ODEs
\(   \label{Fourier}
    \part{\widehat{\varrho}_n}{t} + |q_n|^2 \frac{G(\varrho_0)}{2}
         \left( G(\varrho_0) + 2G'(\varrho_0)\varrho_0 \widehat{W}_n\right) \widehat{\varrho}_n = 0
\)
for the Fourier transformed perturbation $\tilde\varrho = \tilde\varrho(t, \x)$ of $\varrho_0$,
\[
   \widehat \varrho_n : = \int_\Omega \tilde\varrho(\x) e^{-i q_n\cdot \x}\,\d \x, \qquad
   \widehat W_n : = \int_\Omega W(\x) e^{-i q_n\cdot \x}\,\d \x,
\]
with $q_n := 2\pi n$, $n\in \Z^d$. We note that due to the radial symmetry of $W$, all Fourier coefficients $\widehat W_n$ are real.
The zero mode $n=0$ in \eqref{Fourier} is neutral, reflecting conservation of  the total mass.
With the assumption $G(\varrho_0) > 0$ and $G'(\varrho_0) < 0$,
those modes $n\neq 0$ are linearly unstable for which
\(   \label{eq:stability}
    \widehat{W}_n > - \frac{G(\varrho_0)}{2G'(\varrho_0)\varrho_0} \geq 0 \,.
\)
Since $W\in L^1(\Omega)$, we have $\widehat{W}_n \to 0$ as $|n|\to+\infty$ by Riemann--Lebesgue lemma.
Therefore, all modes with $|n|$ larger than a certain threshold are linearly asymptotically stable.

%%%%%%%%%%%%%%%%%%%%%%%%%%%%%%%%%%%%%%%%
\section{Stability analysis for the Vlasov--Fokker--Planck equation}\label{sec:stability}

In this section we derive spatially homogeneous steady states of the Vlasov--Fokker--Planck equation \eqref{eq:FP}
and investigate their linearized stability with respect to spatially inhomogeneous perturbations.
We then use the results of the analysis to gain an insight into the numerical observations for the discrete model \eqref{rho_i}--\eqref{eq:SDE1}
obtained in Sections 6--7 of \cite{EH:2026}, in particular, to explain the coarsening effect induced by the presence of memory.
%dependence of the number and size of clusters on the length of the memory.

%%%%%%%%%%%%%%%%%%%%%%%%%%%%%%%%%%%%%%%%
\subsection{Spatially homogeneous steady states}
\label{subsec:homogeneous-steady-states}
%For simplicity we set $\varepsilon_k := 1$ for all $k\in[K]$ in the sequel.

\begin{remark}
In \cite[Section 5]{EH:2026}, a stationary measure for \eqref{eq:FP} was considered which is concentrated on the lower-dimensional manifold
$\y^k=\alpha_k^{-1}\y^{k+1}$ for $k\in [K-1]$.
This is obtained in the singular relaxation limit of \eqref{eq:FP} as $\varepsilon_k\to 0^+$ for $k\in [K-1]$,
where the variables $\y^1,\dots,y^{K-1}$ relax instantaneously relative to $\y^K$,
which is Gaussian distributed.
In contrast, in this paper we consider the case of finite $\varepsilon_k>0$ for all $k\in[K]$,
where the homogeneous invariant measure of the Vlasov--Fokker--Planck equation is the full centered Gaussian in all internal variables.
%which we derive in the following Lemma.
\end{remark}

To simplify the analysis, %and without significant loss of generality,
we set $\varepsilon_k:=1$ for all $k\in[K]$ in the rest of the paper.

\begin{lemma}\label{lem:steady}
For every constant $\varrho_0>0$ the Vlasov--Fokker--Planck equation \eqref{eq:FPcompact}
%with $\varepsilon_k := 1$, $k\in[K]$,
admits the spatially homogeneous steady state $f_0(\x,\yy) =  \varrho_0 M_0(\yy)$
with $M_0$ given by
\begin{equation}
\label{eq:M0}
    %M_0(\yy) = \frac{1}{(2\pi)^{dK/2}(\det\Sigma)^{d/2}} \exp\left( -\frac12 \yy^T\Sigma^{-1}\yy \right),
   M_0(\yy) = \frac{1}{(2\pi)^{dK/2}(\det\Sigma)^{d/2}} \exp\left( -\frac12 \sum_{\ell=1}^d \yy_\ell^T\Sigma^{-1}\yy_\ell \right),
\end{equation}
with the covariance matrix
\(  \label{eq:Sigma-integral}
   \Sigma = g_0^2\int_0^\infty e^{tA} \ee_K \ee_K^T e^{tA^T}\,\d t,
\)
where we denoted $g_0:=G(\rho_0)>0$, and we recall that $A\in\R^{K\times K}$ is given by \eqref{def:A},
and $\ee_K=(0,\dots,0,1)^T\in\mathbb R^K$ denotes the $K$-th standard basis vector.

In the special case $\alpha_k=\alpha>0$ for all $k\in[K]$, \eqref{eq:Sigma-integral} evaluates to
\begin{equation}
\label{eq:Sigma-explicit-equal-alpha}
   \Sigma_{ij} = g_0^2 \, \frac{(2K-i-j)!} {(K-i)!(K-j)!(2\alpha)^{2K-i-j+1}} \qquad\mbox{for } i,j \in [K].
\end{equation}
\end{lemma}

\begin{proof}
We first note that $M_0=M_0(\yy)$ is a probability density on $\R^{dK}$,
so that $\rho[f_0] = \varrho_0$. %the particle number density of $f_0 =  \rho_0 M_0$ is $\rho_0$.
Moreover, due to the normalization \eqref{eq:Wnorm} we have $W\conv\varrho_0=\varrho_0$,
and therefore the noise amplitude in \eqref{def:L} becomes the constant $g_0=G(\varrho_0)>0$.
Moreover, since $f_0$ is independent of $x\in\Omega$, the transport term $y^1\cdot \nabla_x f_0$ in \eqref{eq:FPcompact} vanishes.
We thus obtain the stationary equation $\mathcal{L}[\varrho_0] M_0 = 0$ with $\mathcal{L}$ given by \eqref{def:L}, i.e.,
\begin{equation*} %\label{eq:stationary-internal-FP}
  -\sum_{k=1}^{K-1}   \nabla_{y^k}\cdot\left(  \left(y^{k+1}-\alpha_k y^k\right)M_0 \right)
  +  \nabla_{y^K}\cdot\left(  \alpha_K y^K M_0  \right)
  +  \frac{g_0^2}{2}\Delta_{y^K}M_0 = 0.
\end{equation*}
This is the stationary Fokker--Planck equation associated with the $d$ identical Ornstein--Uhlenbeck systems
\begin{equation}  \label{eq:OU}
   \d\yy_\ell = \AK\yy_\ell \,\d t + g_0 \ee_K \,\d B_{\ell,t}, \qquad \ell\in [d].
\end{equation}
Since the eigenvalues $-\alpha_1,\dots,-\alpha_K$ of $A$ are all negative, the matrix $A$ is stable
and the Ornstein--Uhlenbeck process \eqref{eq:OU} has the unique invariant centered Gaussian law given by \eqref{eq:M0},
see, e.g., \cite{Mao:2007, Oksendal},
with $\Sigma$ the unique positive definite solution of the Lyapunov equation
\begin{equation}
\label{eq:Lyapunov-Sigma}
   A\Sigma+\Sigma A^T+g_0^2 \ee_K \ee_K^T=0.
\end{equation}
Since the matrix $A$ is stable, the integral \eqref{eq:Sigma-integral} is well defined.
We verify that it solves the Lyapunov equation \eqref{eq:Lyapunov-Sigma}.
Indeed,
\[
   \begin{aligned}
   A\Sigma+\Sigma A^T &= g_0^2\int_0^\infty \left( Ae^{tA}e_Ke_K^T e^{tA^T} +
     e^{tA}e_Ke_K^T e^{tA^T}A^T \right) \d t  \\
    &=
   g_0^2\int_0^\infty \frac{\d}{\d t} \left( e^{tA}e_Ke_K^T e^{tA^T} \right) \d t  \\
   &=
   g_0^2  \left[  e^{tA}e_Ke_K^T e^{tA^T}  \right]_{t=0}^{t=\infty}.
   \end{aligned}
\]
Since $e^{tA}\to0$ as $t\to\infty$, we have
\[
    A\Sigma+\Sigma A^T = -g_0^2 e_Ke_K^T,
\]
so that $\Sigma$ given by \eqref{eq:Sigma-integral} indeed solves \eqref{eq:Lyapunov-Sigma}.
%Equivalently,
%\begin{equation}  \label{eq:Sigma-integral}   \Sigma = g_0^2 \int_0^\infty e^{tA}e_Ke_K^T e^{tA^T},dt . \end{equation}

In the special case $\alpha_k := \alpha>0$ for all $k\in[K]$, we have $\AK=-\alpha I+U$,
where $U$ is the upper shift matrix with elements $U_{k,k+1}=1$ and zero otherwise.
%In particular, all eigenvalues of matrix $\AK$ are equal to $-\alpha$.}
Then $U$ commutes with the identity matrix, and we have
\[
    e^{tA} =  e^{-\alpha t} \, e^{t U} \,.
\]
Moreover, $U$ being upper diagonal, its $K$-th power $U^K$ is the zero matrix, so that
\[
   e^{t U} = \sum_{k=0}^{K-1} \frac{t^k U^k}{k!}.
\]
Consequently, the $(j,K)$-th element is
\(   \label{eq:exptA}
  \left(e^{tA}\right)_{jK} = e^{-\alpha t}\frac{t^{K-j}}{(K-j)!}, \qquad j \in [K].
\)
Substituting this formula into \eqref{eq:Sigma-integral} gives 
\[
   \Sigma_{ij} = g_0^2\int_0^\infty \left(e^{tA}\right)_{iK} \left(e^{tA}\right)_{jK} \d t =
     \frac{g_0^2}{(K-i)! (K-j)!} \int_0^\infty e^{-2\alpha t} \, t^{2K-i-j} \,\d t,
\]
and a straightforward calculation gives \eqref{eq:Sigma-explicit-equal-alpha}.
\end{proof}

%%%%%%%%%%%%%%%%%%%%%%%%%%%%%%%%%%%%%%%%%%
\subsection{Linearized stability analysis of the homogeneous steady state}\label{subsec:linearized-stability}
We now study the linearized stability of the spatially homogeneous steady state
$f_0(\yy) =  \varrho_0 M_0(\yy)$ of \eqref{eq:FPcompact}, with a fixed constant $\varrho_0>0$ and $M_0$ given by \eqref{eq:M0}.
We denote
\[
    g_0:=G(\varrho_0),  \qquad  g_1:=G'(\varrho_0),
\]
and recall that due to the normalization \eqref{eq:Wnorm}, we have $W\conv\varrho_0=\varrho_0$.
We denote $\mathcal L_0 := \mathcal L[\varrho_0]$  the linear operator
\(  \label{def:L0}
   \mathcal L_0 h = -\grad_\yy\cdot\left( A\yy h \right) + \frac{g_0^2}{2}\laplace_{y^K} h,
\)
and observe that, by construction, $\mathcal L_0 M_0=0$.

For $\eta>0$ we consider a total mass preserving perturbation $h=h(t,x,\yy)$ of the homogeneous state $\varrho_0 M_0$,
\(  \label{eq:pert}
   f(t,\x,\yy) = \varrho_0 M_0(\yy)+\eta h(t,\x,\yy) \qquad\mbox{with } \int_\Omega \int_{\R^{dK}} h(t,\x,\yy) \, \d\yy\d\x = 0.
\)
Taylor expansion of the nonlocal diffusion coefficient, using $\rho[f] = \varrho_0 + \eta\rho[h]$, gives
\[
   G(W\conv \rho[f])^2 = g_0^2 + 2\eta g_0g_1 W\conv \rho[h] + \mathcal O(\eta^2).
\]
Substituting this expansion into \eqref{eq:FPcompact} and retaining only terms of order $\eta$, we obtain
\begin{equation} \label{eq:linearized-h}
   \partial_t h + \y^1\cdot\nabla_\x h = \mathcal L_0 h + \rho_0\, g_0\, g_1\, (W\conv \rho[h])\, \Delta_{\y^K}M_0. %+ \mathcal{O}(\eta^2}
\end{equation}
Here we used the fact that $W\conv \rho[h]$ depends only on $\x$.
%Using \eqref{eq:M0}, a simple calculation yields
%$$\Delta_{\y^K}M_0 (\yy) = \left[ \sum_{\ell=1}^d \left( \ee_K^T\Sigma^{-1}\yy_\ell \right)^2 - d \,\ee_K^T\Sigma^{-1} \ee_K \right] M_0 (\yy). $$

We now expand \eqref{eq:linearized-h} into a Fourier series in the $x$-variable on the torus $\Omega$.
Let $q_n:=2\pi n$ with $n\in\mathbb Z^d$ and denote
\[
   %h(t,\x,\yy) = \sum_{n\in\mathbb Z} \widehat h_n(t,\yy) e^{iq_n\cdot \x}, \qquad r(t,\x) = \sum_{n\in\mathbb Z}\widehat r_n(t)e^{iq_n\cdot \x},
   \widehat h_n(t,\yy) := \int_\Omega h(t,\x,\yy) e^{-iq_n\cdot \x} \d\x. %, \qquad    \widehat r_n(t) := \int_{\mathbb R^{dK}}\widehat h_n(t,\yy)\,\d\yy,
\]
%and
%\[ \widehat W_n :=  \int_\Omega W(\x) e^{-iq_n\cdot \x} \d\x, \qquad W\conv r = \sum_{n\in\mathbb Z}\widehat W_n \widehat r_n(t) e^{iq_n\cdot\x}. \]
%Thus each Fourier mode evolves independently.
For $n\in\mathbb Z$, define the operator
\begin{equation}
\label{def:Bn}
   \mathcal B_n := \mathcal L_0-iq_n\cdot \y^1.
\end{equation}
Then a projection of \eqref{eq:linearized-h} onto the $n$-th Fourier mode gives
\begin{equation}
   \label{eq:linearized-Fourier-mode}
   (\partial_t - \mathcal B_n ) \widehat h_n = \rho_0\, g_0\, g_1 \widehat W_n\, \rho[\widehat h_n]\, \Delta_{y^K}M_0.
\end{equation}
The zero mode $n=0$ corresponds to spatially homogeneous perturbations. % with $\rho[\widehat h_0] = 0$.
%Since the total mass is conserved, admissible perturbations in the fixed-mass class must satisfy $\widehat r_0=0$
Spatial inhomogeneities can therefore only arise from modes with $n\neq0$, which we assume in the sequel.
To detect unstable modes,  we look for solutions of \eqref{eq:linearized-Fourier-mode} of the form
\[
   \widehat h_n(t,\yy)=e^{\lambda t}\varphi_n(\yy), %  \qquad \widehat r_n(t)=e^{\lambda t} \psi_n,
\]
where $\lambda\in\C$ with $\operatorname{Re}\lambda >0$.
Denoting $\psi_n := \rho[\varphi_n]$ and substituting into \eqref{eq:linearized-Fourier-mode} gives
\begin{equation}
   \label{eq:eigenvalue-problem}
   (\lambda - \mathcal B_n)\varphi_n = \rho_0 g_0 g_1 \widehat W_n \psi_n \Delta_{y^K}M_0.
\end{equation}
We observe that if $\psi_n=0$, the right-hand side of \eqref{eq:eigenvalue-problem} vanishes and $\lambda$ is an eigenvalue of the operator $\mathcal B_n$,
with an eigenfunction with vanishing particle number density.
Again, such modes cannot give rise to aggregation instabilities in the spatial variable,
and we thus impose $\psi_n\neq0$.
%Such internal (kinetic) modes are not generated by the nonlocal feedback term, i.e., the right-hand side of \eqref{eq:linearized-Fourier-mode}.
%The full linear stability of the $n$-th Fourier mode therefore requires two conditions: first, all roots of the dispersion relation associated with $\psi_n\neq0$ must satisfy $\operatorname{Re}\lambda<0$; second, all eigenvalues of the free kinetic operator $\mathcal B_n$ with eigenfunctions satisfying $P\varphi_n=0$ must also lie in the open left half-plane. For the Ornstein--Uhlenbeck chain with $\alpha_k>0$, the latter modes are expected to be damped by the hypocoercive relaxation of the free kinetic dynamics.
According to Lemma \ref{lem:resolvent} of the Appendix, the resolvent set of the operator $\mathcal B_n$ contains the open right complex half-plane.
Therefore, any $\lambda\in\C$ with $\operatorname{Re}\lambda >0$ belongs to the resolvent set of $\mathcal B_n$
and \eqref{eq:eigenvalue-problem} can be inverted,
\[
   \varphi_n = \rho_0 g_0 g_1 \widehat W_n \psi_n (\lambda-\mathcal B_n)^{-1}\Delta_{y^K}M_0.
\]
Integrating with respect to $\yy\in\R^{dK}$, using $\psi_n := \rho[\varphi_n]$, and dividing by $\psi_n \neq 0$, we arrive at the dispersion relation
\begin{equation}
   \label{eq:dispersion-general}
    \rho_0\, g_0\, g_1 \widehat W_n \left\langle 1, (\lambda-\mathcal B_n)^{-1}\Delta_{y^K}M_0 \right\rangle = 1,
\end{equation}
with the bilinear $L^2(\R^{dK})$-pairing
\(   \label{eq:scalarProd}
   \langle u,v\rangle := \int_{\mathbb \R^{dK}} u(\yy)v(\yy) \,\d \yy.
\)
%The homogeneous steady state \eqref{eq:homogeneous-ansatz} is linearly asymptotically stable if for every $n\neq 0$ all roots $\lambda\in\C$ of \eqref{eq:dispersion-general} have negative real parts,
%and if the remaining spectrum of $\mathcal B_n$ also lies in the open left half-plane. Conversely, if for some $n\neq 0$ the dispersion relation has a root with positive real part,
%then the homogeneous steady state is linearly unstable with respect to perturbations of spatial wave number $q_n$.
With the resolvent formula %(Laplace transform)
\[
   (\lambda-\mathcal B_n)^{-1} = \int_0^\infty e^{-\lambda t}e^{t\mathcal B_n} \,\d t,
\]
we have
\begin{equation}
   \label{eq:Hn-semigroup}
    \left\langle 1, (\lambda-\mathcal B_n)^{-1}\Delta_{y^K}M_0 \right\rangle =
     \int_0^\infty e^{-\lambda t} \left\langle 1, e^{t\mathcal B_n}\Delta_{y^K}M_0 \right\rangle \d t.
\end{equation}
Using the adjoint semigroup and integration by parts, we write
\[
   \left\langle 1, e^{t\mathcal B_n}\laplace_{y^K}M_0 \right\rangle
      = \left\langle e^{t\mathcal C_n}1, \laplace_{y^K}M_0 \right\rangle
      = \left\langle \laplace_{y^K} e^{t\mathcal C_n}1, M_0 \right\rangle,
\]
with the backward Kolmogorov operator $\mathcal C_n$, i.e., formal adjoint of $\mathcal B_n$ with respect to \eqref{eq:scalarProd},
\(   \label{def:Cn}
  \mathcal C_n := (A\yy)\cdot\nabla_\yy + \frac{g_0^2}{2}\Delta_{y^K} - iq_n\cdot y^1.
\)
%where $A$ is the drift matrix of the internal Ornstein--Uhlenbeck chain.
By the Feynman--Kac formula \cite[Section 2.8]{Mao:2007},
\[
   e^{t\mathcal C_n}1(\yy)=\mathbb E \left[ \exp\left(-iq_n\cdot\int_0^t Y_s^1\,ds\right) \Bigg| \, Y_0=\yy \right],
\]
where $(Y_s)_{s\geq 0}$ is the Ornstein--Uhlenbeck process \eqref{eq:OU} started from $Y_0=\yy$.
It is given, componentwise, by the variation-of-constants formula
\[
    Y_s=e^{sA}\yy+g_0\int_0^s e^{(s-\tau)A} \ee_K\,\d B_\tau .
\]
Consequently,
\(  \label{eq:Cj}
   \int_0^t Y_s^1\,\d s = \sum_{j=1}^K C_j(t) y^j + Z_t,\qquad\mbox{with } C_j(t):=\int_0^t(e^{s A})_{1j}\,\d s,
\)
and $Z_t$ is %a centered Gaussian
independent of $\yy$.
%Hence, $$e^{t\mathcal C_n}1(\yy)=c(t,y^1,\dots,y^{K-1})\exp\left(-iq_n\cdot B_K(t)y^K\right),$$
%$with $c$ independent of $y^K$.
We then readily have
\[
   \laplace_{y^K} e^{t\mathcal C_n}1 = -|q_n|^2 C_K(t)^2 e^{t\mathcal C_n}1 .
\]
Finally, we note that $\int_0^t Y_s^1\,\d s$ is a centered Gaussian under the stationary law $M_0$, since
\[
    \mathbb E_{M_0}\int_0^tY_s^1\,\d s = \int_0^t \mathbb E_{M_0} Y_s^1\,\d s   =    0.
\]
Then the Fourier transform formula for a Gaussian measure \cite[Appendix A]{Oksendal} gives
\[
   \left\langle e^{t\mathcal C_n}1,M_0\right\rangle
   =  \mathbb E_{M_0} \left[ \exp\left(-iq_n\cdot\int_0^t Y_s^1\,\d s\right) \right]
   =  \exp\left(-\frac{|q_n|^2}{2}V(t)\right),
\]
with the variance
\(   \label{def:V}
   V(t) := \operatorname{Var}_{M_0} \left( \int_0^t Y^1(s)\,\d s \right),
\)
where $(Y_s)_{s\geq 0}$ is one spatial component of the Ornstein--Uhlenbeck process \eqref{eq:OU}
with $Y_0$ distributed according to $M_0$.
Inserting into \eqref{eq:Hn-semigroup}, the dispersion relation \eqref{eq:dispersion-general} takes the form
%can be equivalently written as
\begin{equation}
\label{eq:disp0}
    -\rho_0\, g_0 \, g_1 \widehat W_n |q_n|^2 \int_0^\infty e^{-\lambda t} C_K(t)^2 \exp\left( -\frac{|q_n|^2}{2}V(t) \right) \,\d t =  1.
\end{equation}
We conclude that the Fourier mode $n\in\Z^d$ induces a spatially inhomogeneous instability
if \eqref{eq:disp0} admits a root $\lambda\in\C$ with positive real part.

%%%%%%%%%%%%%%%%%%%%%%%%%%%%%%%%%%%%%%%%%%
\subsection{Analysis of the dispersion relation for equal $\alpha_k$}\label{subsec:dispersionAnalysis}
We now analyze \eqref{eq:disp0} in the special case $\alpha_k=\alpha>0$ for all $k\in[K]$.
Using formula \eqref{eq:exptA} for $\left(e^{sA}\right)_{1K}$,
we write $C_K=C_K(t)$, defined in \eqref{eq:Cj}, in the form
\begin{equation}
   \label{eq:CK-equal-alpha}
     C_K(t) = \frac{1}{(K-1)!} \int_0^t e^{-\alpha s}s^{K-1} \,\d s
        = \frac{1}{\alpha^K} \left[ 1-e^{-\alpha t} \sum_{j=0}^{K-1}\frac{(\alpha t)^j}{j!} \right].
\end{equation}

Next we prove a useful scaling relation for $V=V(t)$ given by \eqref{def:V}.

\begin{lemma}
Denote $V_\alpha=V_\alpha(t)$ the variance given by \eqref{def:V},
with $(Y_s)_{s\geq 0}$ being one spatial component of
the Ornstein--Uhlenbeck process \eqref{eq:OU} with parameters $\alpha>0$ and $g_0>0$.
%and with $Y_0$ distributed according to $M_0$.
Then
\(  \label{eq:Valpha}
   V_{\alpha}(t)  =  g_0^2\alpha^{-(2K+1)} \widetilde V(\alpha t),
\)
where
\(   \label{def:widetildeV}
   \widetilde V(\tau) =  \operatorname{Var}_{\widetilde M_0}  \left(  \int_0^{\tau}\widetilde Y_s^1\, \d s \right),
\)
and $(\widetilde Y_s)_{s\geq 0}$ is one spatial component of the Ornstein--Uhlenbeck process \eqref{eq:OU}
with parameters $\alpha:=1$ and $g_0:=1$.
\end{lemma}

\begin{proof}
%Let us denote $(\widetilde Y_s)_{s\geq 0}$ the Ornstein--Uhlenbeck process \eqref{eq:OU} with parameters $\alpha:=1$ and $g_0:=1$.
We claim that the processes $Y_t$ and $\widetilde Y_{\alpha t}$ are related by the following scaling relation,
\(   \label{eq:scaling}
    Y_t^k  = g_0\alpha^{-(K-k+1/2)}\widetilde Y_{\alpha t}^k \qquad\mbox{for } k\in [K].
\)
%with $(Y_s)_{s\geq 0}$ being the Ornstein--Uhlenbeck process \eqref{eq:OU} with parameter $\alpha>0$.
Indeed, for $k\in[K-1]$,
\[
\begin{aligned}
    \d Y_t^k
    &=   g_0\alpha^{-(K-k+1/2)}\,\d \widetilde Y_{\alpha t}^k \\
    &=   g_0\alpha^{-(K-k+1/2)}  \left(\widetilde Y_{\alpha t}^{k+1}-\widetilde Y_{\alpha t}^k \right) \,  \alpha \d t \\
 %   &= \left(  g_0\alpha^{-(K-(k+1)+1/2)}\widetilde Y_{\alpha t}^{k+1} -    \alpha g_0\alpha^{-(K-k+1/2)}\widetilde Y_{\alpha t}^k    \right)\,d t \\
    &=   \left( Y_t^{k+1} - \alpha Y_t^k \right) \,\d t .
\end{aligned}
\]
For $k=K$, using $\d B_{\alpha t} = \sqrt{\alpha}\,\d B_t$, we get
\[
\begin{aligned}
    \d Y_t^K
    &=  g_0\alpha^{-1/2}\,\d\widetilde Y_{\alpha t}^K \\
    &=   g_0\alpha^{-1/2}
    \left(  -\widetilde Y_{\alpha t}^K\,\alpha\,\d t  +  \d B_{\alpha t}  \right) \\
    &=   -\alpha Y_t^K\,\d t+g_0\,\d B_t,
\end{aligned}
\]
%In particular, $Y_t^1  =    g_0\alpha^{-(K-1/2)}\widetilde Y_{\alpha t}^1$.
so that \eqref{eq:scaling} holds. Using it with $k=1$, we have
\[
    \int_0^t Y_s^1\,\d s =   g_0\alpha^{-(K-1/2)}   \int_0^t\widetilde Y_{\alpha s}^1\,\d s 
    = g_0\alpha^{-(K+1/2)}  \int_0^{\alpha t}\widetilde Y_\tau^1\, \d\tau.
\]
Taking the variance gives
\[
    V_{\alpha}(t)  =  \operatorname{Var}_{M_0} \left( \int_0^t Y^1(s)\,\d s \right)
      = g_0^2\alpha^{-(2K+1)}  \operatorname{Var}_{\widetilde M_0}
      \left(  \int_0^{\alpha t}\widetilde Y_\tau^1\, \d\tau \right),
\]
where $\widetilde M_0$ is the invariant Gaussian law of the system with $\alpha=1$, $g_0=1$.
This concludes \eqref{eq:Valpha}.
\end{proof}

We now rewrite the dispersion relation \eqref{eq:disp0} in the form
\(  \label{eq:general-K-dispersion-alpha}
   -\rho_0 g_0 g_1\widehat W_n |q_n|^2 I_n(\alpha;\lambda) = 1,
\)
with
\[
   I_n(\alpha;\lambda) := \int_0^\infty e^{-\lambda t} C_K(t)^2 \exp\left( -\frac{|q_n|^2}{2}V_\alpha(t) \right) \,\d t.
\]
We also rewrite $C_K=C_K(t)$ given by \eqref{eq:CK-equal-alpha} as $C_K(t) = \alpha^{-K} F_K(\alpha t)$ with
\(   \label{def:FK}
   F_K(\tau) :=  \frac{1}{(K-1)!} \int_0^{\tau} e^{-s}s ^{K-1} \,\d s.
\)
%The function $F_K$ is increasing, satisfies $F_K(0)=0$, and converges to $1$ as $\tau\to+\infty$.
Then the change of coordinates $\tau = \alpha t$, using \eqref{eq:Valpha}, gives
\(  \label{eq:InK-alpha-scaled}
    I_n(\alpha;\lambda)  =  \alpha^{-(2K+1) }\int_0^\infty e^{-\frac{\lambda \tau}{\alpha}}  F_K(\tau)^2 \exp\left(-\frac{\gamma_n}{\alpha^{2K+1}} \widetilde V(\tau) \right) \,\d\tau
    \qquad\mbox{with }
    \gamma_n := \frac{|q_n|^2 g_0^2}{2}.
\)

\begin{lemma}\label{lem:I_lambda}
For any $\alpha>0$, we have $I_n(\alpha;\lambda) <+\infty$ for every $\lambda\in\C$ with $\Re\lambda > -\gamma_n \alpha^{-2K}$.
Moreover, for any fixed $\alpha>0$, $I_n(\alpha;\lambda)$ is a decreasing function of the real variable $\lambda\in (-\gamma_n\alpha^{-2K}, +\infty)$, and
\[
   \lim_{\lambda\to -\gamma_n\alpha^{-2K}} I_n(\alpha;\lambda) = +\infty, \qquad \lim_{\lambda\to+\infty} I_n(\alpha;\lambda) =  0.
\]
\end{lemma}

\begin{proof}
We show that
\[
   F_K(\tau)^2 \exp\left(-\frac{\gamma_n}{\alpha^{2K+1}} \widetilde V(\tau) \right) \asymp \exp\left( -\frac{\gamma_n}{\alpha^{2K+1}} \tau\right) \qquad\mbox{as } \tau\to+\infty.
\]
From \eqref{def:FK} it immediately follows that $0 < F_K(\tau) <1$ for $\tau>0$, and $\lim_{\tau\to+\infty} F_K(\tau) = 1$.
It remains to prove that
\(   \label{eq:wVas}
   \widetilde V(t) = t + \mathcal{O}(1) \qquad \mbox{as } t\to+\infty.
\)
Let $(\widetilde Y_s)_{s\geq 0}$ be the stationary Ornstein--Uhlenbeck process \eqref{eq:OU} with parameters $\alpha:=1$ and $g_0:=1$.
As a consequence of the variation-of-constants formula \cite[Section 3.3]{Mao:2007},
it admits the moving-average representation
\begin{equation*}
    \widetilde Y_t = \int_{-\infty}^t e^{A(t-s)}\ee_K\,\d B_s,
\end{equation*}
where $B=(B_s)_{s\in\R}$ is the two-sided Brownian motion.
Formula \eqref{eq:exptA} with $\alpha=1$ and $j=1$ gives
\begin{equation*}
    \ee_1^T e^{A t}\ee_K  =e^{-t}\ee_1^T e^{Ut}\ee_K
    =\frac{e^{-t}t^{K-1}}{(K-1)!}    =F_K'(t),
\end{equation*}
where $F_K=F_K(t)$ is given by \eqref{def:FK}.
Consequently,
\begin{equation*}
    \widetilde Y^1_t = \int_{-\infty}^t F_K'(t-s)\,\d B_s.
\end{equation*}
The stationary covariance function of $\widetilde Y_t$ is therefore
\begin{equation}  \label{eq:RK}
    R_K(t):=\operatorname{Cov}_{\widetilde M_0}(\widetilde Y_t^1, \widetilde Y_0^1)
      =\int_0^\infty F_K'(v)F_K'(v+t)\,\d v  \qquad \mbox{for } t\geq 0.
\end{equation}
By stationarity, it follows that
\(
    \widetilde V(t) = \operatorname{Var}_{\widetilde M_0}  \left(  \int_0^{t}\widetilde Y_\tau^1\, \d\tau \right)
     = \int_0^t \int_0^t \operatorname{Cov}_{\widetilde M_0}(\widetilde Y_s^1, \widetilde Y_r^1) \,\d r\d s  \nonumber\\
     = \int_0^t \int_0^t  R_K(|s-r|) \,\d r\d s   =  2\int_0^t(t-s)R_K(s)\,\d s.      \label{eq:wV}
\)
%Moreover,
%\[   \widetilde V'(t)=2\int_0^tR_K(s)\,\d s. \]
%In particular, $\widetilde V'(t)>0$ for every $t>0$.
Since $R_K(s)$ decays exponentially as $s\to\infty$, it follows that
\[
    \widetilde V(t)  =  2t\int_0^\infty R_K(s)\,\d s + \mathcal O(1)    \qquad \mbox{as } t\to\infty.
\]
Recalling \eqref{eq:RK}, we have
\(
   2 \int_0^\infty R_K(s)\,\d s &=& 2 \int_0^\infty \int_0^\infty F_K'(v)F_K'(v+s)\,\d v\,\d s  \nonumber \\
     &=& \int_0^\infty\int_0^\infty  F_K'(v) F_K'(u) \,\d v\,\d u =1,
    %F_K'(v)=\frac{e^{-v}v^{K-1}}{(K-1)!}.
    \label{eq:intRK}
\)
where we used the fact that $F_K'$ is a probability density on $(0,\infty)$.
This proves \eqref{eq:wVas}.

Consequently, the integral in \eqref{eq:InK-alpha-scaled} is finite if and only if $\Re\lambda > -\gamma_n \alpha^{-2K}$,
and we have
\[
   \lim_{\lambda\to -\gamma_n\alpha^{-2K}} I_n(\alpha;\lambda) = +\infty.
\]
Also, obviously, $I_n(\alpha;\lambda)$ is a decreasing function of $\lambda$ on its domain.
Finally, the claim $\lim_{\lambda\to+\infty} I_n(\alpha;\lambda) =  0$ follows immediately by the
dominated convergence theorem.
\end{proof}

\begin{corollary}\label{corr:1}
Let $n\neq 0$ and $\widehat W_n >0$. 
Then for every $\alpha>0$ the dispersion relation \eqref{eq:general-K-dispersion-alpha}
has a unique real solution $\lambda_\alpha > -\gamma_n \alpha^{-2K}$. 
Moreover, $\lambda_\alpha\in\R$ is the rightmost root, %of \eqref{eq:general-K-dispersion-alpha},
i.e., all other solutions $\lambda\in\C$ verify $\Re \lambda \leq \lambda_\alpha$.
\end{corollary}

\begin{proof}
The unique real solvability of \eqref{eq:general-K-dispersion-alpha} follows directly from Lemma \ref{lem:I_lambda},
recalling that $g_0>0$, $g_1 <0$.
Moreover, if $\lambda\in\C$ is any other solution, then
\[
   I_n(\alpha;\lambda_\alpha) = | I_n(\alpha;\lambda) | \leq I_n(\alpha;\Re\lambda),
\]
and since $s \mapsto I_n(\alpha;s)$ is decreasing on the real axis, we have $\lambda_\alpha \geq \Re\lambda$.
%Moreover, if $\operatorname{Im}\lambda\neq0$, the inequality is strict, because the phase factor $e^{-t i\operatorname{Im}\lambda}$ is not constant on the support of the integrand.
\end{proof}

Corollary \ref{corr:1} implies that, for \(n\neq0\) such that $\widehat W_n>0$ for any $\alpha>0$,
the spatially inhomogeneous Fourier mode $n\neq 0$ is linearly unstable
if and only if the unique real root of \eqref{eq:general-K-dispersion-alpha} is positive.
%\cJH{eventually discuss all three alternatives $\lambda_\alpha$ positive, zero, negative}
The neutral threshold for overall linear instability of the constant steady state of \eqref{eq:FP}
is obtained by setting $\lambda=0$, i.e., looking for $\alpha>0$ such that
\(   \label{eq:dispI0n}
      -\rho_0 g_0 g_1\widehat W_n |q_n|^2 I_n(\alpha;0) = 1.
\)
We therefore study the properties of the function
\(   \label{def:I0alpha}
   I^0_n(\alpha) := \frac{1}{\alpha^{2K+1}} \int_0^\infty   F_K(\tau)^2 \exp\left(-\frac{\gamma_n}{\alpha^{2K+1}} \widetilde V(\tau) \right) \,\d\tau,
\)
where we recall that $\gamma_n = \frac{|q_n|^2 g_0^2}{2}$ and $F_K=F_K(\tau)$ is given by \eqref{def:FK}.

\begin{lemma}\label{lem:I0monotone}
Let $q_n\neq0$ and $g_0>0$. Then the function $I_n^0=I_n^0(\alpha)$ given by \eqref{def:I0alpha} is strictly increasing in $\alpha\in (0,+\infty)$.
%Moreover,
%\[ \lim_{\alpha\to0+} I_n^0(\alpha) (\alpha) = 0, \qquad \lim_{\alpha\to+\infty} I_n^0(\alpha) (\alpha) = \frac{2}{|q_n|^2 g_0^2}. \]
\end{lemma}

\begin{proof}
We show that the function
\begin{equation*}
    \tau  \mapsto \frac{F_K(\tau)^2}{\widetilde V'(\tau)}
\end{equation*}
is strictly increasing for all $\tau>0$.
With \eqref{eq:RK} we have
\begin{align*}
    \frac{R_K(\tau)}{F_K'(\tau)}
    &=\int_0^\infty F_K'(v)\frac{F_K'(v+\tau)}{F_K'(\tau)}\,\d v \notag\\
    &=\int_0^\infty F_K'(v)e^{-v}
    \left(1+\frac{v}{\tau}\right)^{K-1}\,\d v.
\end{align*}
For every fixed $v>0$, the function $\tau \mapsto \left(1+\frac{v}{\tau}\right)^{K-1}$ is nonincreasing on $(0,\infty)$.
Therefore, since $F_K'(\tau) = \frac{e^{-\tau}\tau^{K-1}}{(K-1)!} >0$ for $\tau>0$,
the fraction $R_K/F_K'$ is nonincreasing on $(0,\infty)$, and thus, for any $0<s\leq \tau$,
\begin{equation*}
    \frac{R_K(s)}{F_K'(s)}  \geq  \frac{R_K(\tau)}{F_K'(\tau)},
\end{equation*}
Multiplying by $F_K'(s)$ and integrating over $(0,\tau)$ gives
\begin{equation*}
   \frac{\widetilde V'(\tau)}{2} = \int_0^\tau R_K(s)\,\d s  \geq \frac{R_K(\tau)}{F_K'(\tau)}  \int_0^\tau F_K'(s)\,\d s
      = \frac{R_K(\tau)}{F_K'(\tau)}F_K(\tau),
\end{equation*}
where we took the derivative in \eqref{eq:wV} to obtain the first equality.
Therefore,
\begin{equation}  \label{eq:gS-FR}
    F_K'(\tau) \frac{\widetilde V'(\tau)}{2}  \geq F_K(\tau)R_K(\tau).
\end{equation}
Using $\widetilde V'' = 2R_K$, we calculate
\[
    \frac{\d}{\d \tau}\left(\frac{F_K(\tau)^2}{\widetilde V'(\tau)}\right)   =  \frac{2 F_K(\tau)\bigl(F_K'(\tau) \widetilde V'(\tau) - F_K(\tau)R_K(\tau)\bigr)} {\widetilde V'(\tau)^2}.
\]
By \eqref{eq:gS-FR},
\[
    F_K'(\tau) \widetilde V'(\tau)-F_K(\tau)R_K(\tau)  \geq F_K(\tau)R_K(\tau)>0.
\]
Consequently,
\begin{equation}    \label{eq:FK-V-monotone}
    \frac{\d}{\d \tau}\left(\frac{F_K(\tau)^2}{\widetilde V'(\tau)}\right)>0
    \qquad\text{for every } \tau>0.
\end{equation}
Recalling that $\widetilde V'(\tau) =2\int_0^\tau R_K(s)\,\d s > 0$,
we perform the increasing change of variable $u := \alpha^{-(2K+1)} \widetilde V(\tau)$ in \eqref{def:I0alpha}.
Due to \eqref{eq:wVas}, $\widetilde V$ maps $(0,\infty)$ to itself, consequently,
\(
     I_n^0(\alpha)  &=&  \frac{1}{\alpha^{2K+1}}  \int_0^\infty F_K(\tau)^2 \exp\left(-\frac{\gamma_n}{\alpha^{2K+1}} \widetilde V(\tau) \right) \,\d\tau \nonumber \\
      &=& \int_0^\infty \frac{F_K(\tau_\alpha(u))^2}{\widetilde V'(\tau_\alpha(u))} e^{ -\gamma_n u} \,\d u,
        \label{eq:I0}
\)
where we denoted $\tau_\alpha(u) := \widetilde V^{-1} \left( \alpha^{2K+1} u \right)$.
Then, since $\widetilde V^{-1}$ is strictly increasing, \eqref{eq:FK-V-monotone} implies that
\[
   \alpha \mapsto \frac{F_K(\tau_\alpha(u))^2}{\widetilde V'(\tau_\alpha(u))}
\]
is strictly increasing for all $u>0$. We conclude that $I_n^0(\alpha)$ is strictly increasing in $\alpha\in (0,\infty)$.
\end{proof}

\begin{lemma}\label{lem:I0limits}
Let $q_n\neq0$ and $g_0>0$. Then the function $I_n^0(\alpha)=I_n^0(\alpha)$ given by \eqref{def:I0alpha} satisfies
\begin{equation}
    \lim_{\alpha\to0^+} I_n^0(\alpha) =0,  \qquad  \lim_{\alpha\to+\infty}I_n^0(\alpha) = \gamma_n^{-1}.  %\frac{2}{|q_n|^2g_0^2}.
\end{equation}
\end{lemma}

\begin{proof}
Note that, as $\tau\to0^+$,
\begin{equation*}
    F_K(\tau) = \frac{\tau^K}{K!} + \mathcal O(\tau^{K+1}),
    \qquad
%Moreover, since $R_K$ is continuous at zero,
    \widetilde V'(\tau)
    =2\int_0^\tau R_K(s)\,\d s
    =2R_K(0)\tau + o(\tau).
\end{equation*}
Therefore, noting that $R_K(0) = \operatorname{Var}_{M_0} (\widetilde Y_0^1) > 0$, we have
\[
    \frac{2F_K(\tau_\alpha)^2}{\widetilde V'(\tau_\alpha)} = \mathcal O \left( \tau_\alpha^{2K-1} \right)
      =  \mathcal O \left( \left(\alpha^{2K+1} u\right)^\frac{2K-1}{2} \right),
\]
where we used
$$\tau_\alpha = \tau_\alpha(u) = \widetilde V^{-1} \left( \alpha^{2K+1} u\right),$$
and $\widetilde V^{-1}(t) = \mathcal O(t^{1/2})$ as $t\to 0^+$.
Then, the claim $\lim_{\alpha\to0^+} I_n^0(\alpha) =0$ follows from \eqref{eq:I0} by the dominated convergence theorem,
where the integrable dominant can be chosen as the integrand evaluated at $\alpha=1$.

To evaluate the limit as $\alpha\to+\infty$, we note that $\lim_{t\to+\infty}F_K(t)=1$, and, by \eqref{eq:intRK},
\[
    \lim_{t\to+\infty}\widetilde V'(t) = 2\int_0^\infty R_K(s)\,\d s = 1.
\]
Consequently,
\[
    \lim_{\tau\to+\infty} \frac{F_K(\tau)^2}{\widetilde V'(\tau)} = 1.
\]
Noting that $\widetilde V^{-1}(t) = t +\mathcal O(1)$ as $t\to+\infty$, %see "Explanation about $\widetilde V^{-1}(t)$"
we have $\tau_\alpha(u) = \mathcal O(\alpha^{2K+1} u)$,
and using \eqref{eq:I0}, we conclude
\[
   \lim_{\alpha\to+\infty} I_n^0(\alpha) = \int_0^\infty e^{-\gamma_n u} \,\d u = \gamma_n^{-1}.
\]
\end{proof}

%We observe that the limits of $I_n^0(\alpha)$ as $\alpha\to 0^+$ and $\alpha\to+\infty$ are independent of $K$.
Combining the results of Lemmas \ref{lem:I0monotone} and \ref{lem:I0limits}, we observe that if
\(  \label{eq:alphaCrit}
   -2\rho_0 \frac{g_1}{g_0} \widehat W_n > 1,
\)
then for each $K\in\N$ there exists a unique critical value $\bar\alpha=\bar\alpha(K,n)> 0$ such that the dispersion relation
\eqref{eq:general-K-dispersion-alpha} with $\alpha = \bar\alpha$ has the root $\lambda=0$.
This critical value is found as the solution of \eqref{eq:dispI0n}, and its uniqueness
follows from the monotonicity established by Lemma \ref{lem:I0monotone}.
If $\alpha<\bar\alpha$, then all roots $\lambda\in\C$ of the dispersion relation \eqref{eq:general-K-dispersion-alpha}
have negative real parts and, consequently, the Fourier mode $n\in\N$ is linearly asymptotically stable.
On the other hand, if $\alpha>\bar\alpha$, then there exists a root with positive real part
and the Fourier mode is linearly unstable. %For $\alpha=\bar\alpha$ the mode is neutral.
Let us note that the condition \eqref{eq:alphaCrit} is independent of $K\in\N$ and
is identical to the instability condition \eqref{eq:stability} for the system without memory (we remind that $g_1<0$).
Therefore, short memory $\alpha>\bar\alpha$ leaves the $n$-th Fourier mode unstable,
while long memory $\alpha<\bar\alpha$ stabilizes it.

In the next lemma we show that if any Fourier mode in the model with memory is unstable,
then the given homogeneous steady state cannot be stable in the memoryless model
because \eqref{eq:alphaCrit} holds at least for $|n|=1$.

\begin{lemma}\label{lem:stability}
Let $W=W(x)$ be radially symmetric and nonincreasing with respect to $|x|$.
If for some $K\in\N$, $\alpha>0$ and a Fourier mode $n\neq 0$ the dispersion relation \eqref{eq:general-K-dispersion-alpha}
has a root $\lambda\in\C$ with $\Re\lambda \geq 0$, then
\(  \label{eq:instab}
   -2\rho_0 \frac{g_1}{g_0} \widehat W_\mathbf{1} > 1,
\)
where $\mathbf{1}$ is any unit vector in $\Z^d$.
\end{lemma}

\begin{proof}
We first establish the following fact: If $h: [-1/2,1/2] \to[0,\infty)$ is even and nonincreasing on $[0,1/2]$, then
\(  \label{eq:fact}
   \left|\widehat h_k\right|\leq \widehat h_1 \qquad\text{for every } k\in\mathbb Z\setminus\{0\}.
\)
Indeed, %up to an additive constant that does not affect the nonzero Fourier coefficients,
$h=h(s)$ admits the layer-cake representation
\[
   h(s)= h(1/2) + \int_{(0,1/2]} \chi_{\{|s|<r\}} \,\d\mu(r)
\]
for some nonnegative Lebesgue--Stieltjes measure $\mu$.
The $k$-th Fourier coefficient of the characteristic function $\chi_{\{|s|<r\}}$ as a function of $s\in [-1/2,1/2]$ is
${\sin(2\pi kr)}/{\pi k}$. Consequently,
\[
  \left| \widehat h_k \right| =
  \left| \int_{(0,1/2]} \frac{\sin(2\pi kr)}{\pi k}\,\mathrm d\mu(r) \right| \leq
  \int_{(0,1/2]} \frac{\sin(2\pi r)}{\pi}\,\mathrm d\mu(r) = \widehat h_1,
\]
where we used the inequality $|\sin(k\theta)|\leq |k|\sin\theta$ for $\theta\in [0,\pi]$,
established in Lemma \ref{lem:technical} of the Appendix.

We then prove that $|\widehat W_n|\leq \widehat W_\mathbf{1}$ for every $n\neq 0$.
We first note that due to the radial symmetry of $W$, the Fourier coefficient $\widehat W_\mathbf{1}$
does not depend on the particular choice of the unit vector $\mathbf{1}\in \Z^d$.
Consequently, without loss of generality, we assume that $n_1 \neq 0$,
and choose $\mathbf{1} := (1,0,\dots,0)\in\Z^d$.
For each fixed $\tilde x\in [-1/2, 1/2]^{d-1}$ define
\[
   w(x_1; \tilde x) := W(x_1,\tilde x) \qquad\mbox{for } x_1\in [-1/2, 1/2].
\]
We have
\begin{align*}
   |\widehat W_n| &= \left| \int_{\tilde x}  \int_{-1/2}^{1/2} W(x_1,\tilde x) e^{-2\pi i n_1 x_1} e^{-2\pi i \tilde n\cdot\tilde x} \,\d x_1 \d \tilde x \right| \\
    &\leq \int_{\tilde x} \left| \int_{-1/2}^{1/2} W(x_1,\tilde x) e^{-2\pi i n_1 x_1}\,\d x_1 \right|\d \tilde x \\
   &= \int_{\tilde x} \left| \widehat w_{n_1} (\,\cdot\; ; \tilde x) \right| \d \tilde x.
\end{align*}
Using \eqref{eq:fact} with $h:=w(\,\cdot\,; \tilde x)$ and $k:=n_1\neq 0$,
noting that $h$ is even and nonincreasing, we have
$\left|\widehat  w_{n_1} (\,\cdot\; ; \tilde x) \right| \leq \widehat  w_{1} (\,\cdot\, ; \tilde x)$
for all $\tilde x\in [-1/2, 1/2]^{d-1}$.
Consequently,
\[
   |\widehat W_n| \leq \int_{\tilde x} \widehat w_{1} (\,\cdot\, ; \tilde x)  \d \tilde x
   =   \int_{\tilde x} \int_{-1/2}^{1/2}  W(x_1,\tilde x) \cos(2\pi x_1) \, \d x_1 \d \tilde x = \widehat W_\mathbf{1}.
\]

Now, if the dispersion relation \eqref{eq:general-K-dispersion-alpha} has a root $\lambda\in\C$ with $\Re\lambda >0$ for some $n\neq 0$,
then the inequality
\[
   \left| I_n(\alpha;\lambda) \right| \leq I_n(\alpha; \Re \lambda) \leq  I_n(\alpha;0)  < \gamma_n^{-1},
\]
provided by Lemmas \ref{lem:I_lambda}, \ref{lem:I0monotone} and \ref{lem:I0limits}, together with $g_1<0$, gives
\[
     1 = - 2\rho_0 \frac{g_1}{g_0}  \left| \widehat W_n \right| \gamma_n  \left| I_n(\alpha;\lambda) \right| 
     < -2 \rho_0 \frac{g_1}{g_0}  \left| \widehat W_n \right| 
      \leq -2 \rho_0 \frac{g_1}{g_0} \widehat W_\mathbf{1},
\]
which is \eqref{eq:instab}.
\end{proof}

Lemma \ref{lem:stability} implies that a homogeneous state which is stable in the memoryless model
cannot be destabilized by introducing memory, for any $K\in\mathbb N$ and $\alpha>0$.
Memory may only stabilize an otherwise unstable homogeneous state,
but it cannot create an overall instability that is absent without memory.
However, if $\widehat W_n < 0$, memory can destabilize a particular Fourier mode $n\neq 0$ that was stable
in the memoryless model.
However, Lemma \ref{lem:stability} states that if this happens, then \eqref{eq:instab} holds,
i.e., the homogeneous state is already unstable in the memoryless model.
We give a particular example in Section \ref{sec:destab}.

%%%%%%%%%%%%%%%%%%%%%%%%%%%%%%%%
\section{Existence of spatially inhomogeneous steady states}\label{sec:Rabinowitz}
In this section we employ the Crandall--Rabinowitz bifurcation theorem
\cite{Rabinowitz} to prove, under suitable assumptions, the existence of
spatially inhomogeneous stationary solutions of the kinetic Fokker--Planck
equation. These solutions are parametrized by their total mass
and bifurcate from the homogeneous branch at values for which the linearized stationary
operator has a simple zero eigenvalue.

We consider the one-dimensional setting $d=1$, which simplifies the analysis slightly.
Namely, on the torus the stationary problem is invariant under spatial translations.
In 1D the critical Fourier mode gives rise to two phase directions corresponding to sine and cosine modes.
To recover a one-dimensional kernel, as required by the Crandall--Rabinowitz theorem, we
work in the even subspace associated with the reflection $(x,\yy)\mapsto(-x,-\yy)$.
This restriction is compatible with the equation because the sampling
kernel $W$ is assumed to be even. 
In higher spatial dimensions the multiplicities of critical wave vectors
stemming from translation invariances in all $d$ directions
need to be treated by imposing additional symmetries.

We recall the definition \eqref{def:I0alpha} of $I_n^0(\alpha)$, which we here write, for fixed $\alpha>0$, as a function of $\varrho>0$, i.e.,
%$I_0^k=I_0^k(\rho)$ reads
\[  %\label{I0rho}
   I_n^0(\alpha)(\varrho) = \frac{1}{\alpha^{2K+1}} \int_0^\infty 
      F_K(\tau)^2 \exp\left(-\frac{|q_n|^2 G(\varrho)^2}{2\alpha^{2K+1}} \widetilde V(\tau) \right) \,\d\tau.
\]
Note that $F_K=F_K(\tau)$ given by \eqref{def:FK} and $\widetilde V = \widetilde V(\tau)$ given by \eqref{def:widetildeV} do not depend on $\varrho$.

\begin{lemma} %[Existence of spatially inhomogeneous kinetic steady states]
\label{lem:kinetic-inhomogeneous-steady-states}
%Let $W\in L^1(\Omega)$ be even, nonnegative, and normalized as \eqref{eq:Wnorm}.
Let $G\in C^2(0,\infty)$ and for a fixed $\varrho_0>0$ assume that
\(  \label{ass:trans}
   \tot{}{\varrho} \left[ \varrho \frac{G'(\varrho)}{G(\varrho)} \right]_{\varrho=\varrho_0} < 0.
\)
As before, denote $g_0:=G(\varrho_0)$, $g_1:=G'(\varrho_0)$, and
assume that there exists an integer $k\geq1$ such that
\(   \label{ass:kyes}
    - \varrho_0 g_0 g_1 \widehat W_{k} |q_k|^2 I_0^k(\varrho_0) = 1,
\)
while for all $n\geq 1$ such that $n\neq k$,
\(     \label{ass:nno}
   -  \varrho_0 g_0 g_1 \widehat W_{n} |q_n|^2 I_n^0(\alpha)(\varrho_0)  \neq 1.
\)

Then a branch of spatially inhomogeneous stationary solutions bifurcates from the homogeneous state $f_0(\yy) = \varrho_0 M_0(\yy)$,
where $M_0=M[\varrho_0]$ is given by \eqref{eq:M0}--\eqref{eq:Sigma-integral}.
In particular, there exist $\delta>0$ and a curve
$s\mapsto (\varrho_s,f_s)$ %for $|s|<\delta$,
such that %$\varrho_0=\varrho_0$ and $f_0=f_0$
for every $|s|<\delta$, $f_s$ is a stationary solution of the nonlinear kinetic Fokker--Planck equation \eqref{eq:FPcompact}
with total mass $\varrho_s$, given by
\[
    f_s(x,\yy) = \varrho_s M[{\varrho_s}](\yy) +  s\Re\left( e^{iq_k x} \Phi_k(\yy)\right) + o(s),
\]
where
\(   \label{def:Phik}
    \Phi_k(\yy) =  - \varrho_0 g_0 g_1 \widehat W_{k}   \mathcal B_{k}^{-1}\Delta_{y^K}M[\varrho_0],
\)
and $\mathcal B_n\varphi  =  \mathcal L[\varrho_0]\varphi - iq_n y^1\varphi$,
with $\mathcal L[\varrho_0]$ defined in \eqref{def:L}.
Moreover, each $f_s$ is spatially inhomogeneous, with density
\[ 
   \rho[f_s](x) = \varrho_s + s\cos(q_k x) + o(s).
\]
%and therefore $\varrho_{f_s}$ is nonconstant for $s\neq0$ sufficiently small.
\end{lemma}

\begin{proof}
The stationary Fokker--Planck equation \eqref{eq:FPcompact} reads
\(   \label{eq:FPstat}
   \mathcal{L}[\rho[f]] f - y^1\cdot\grad_x f = 0,
\)
with $\rho[f]$ given by \eqref{def:rho}.
We define the mapping
\[
   \mathcal F(\varrho,h) := \mathcal L[\rho[f]] f  -  y^1\cdot\grad_x f \qquad\mbox{with } f:= \varrho M[\varrho]+h,
\]
with $M[\varrho]$ given by \eqref{eq:M0}--\eqref{eq:Sigma-integral}.
We verify that $\mathcal F$ satisfies the assumptions of the Crandall--Rabinowitz bifurcation theorem \cite{Rabinowitz}.

For the constant $\varrho_0>0$, Lemma \ref{lem:steady} states that $\varrho_0 M[\varrho_0]$ is the homogeneous steady state of \eqref{eq:FP} with mass $\varrho_0$. Therefore $\mathcal F(\varrho_0,0)=0$.

Let us now denote $M_0:=M[\varrho_0]$.
We shall work in the subspace $\mathcal H_{0}^{\mathcal R}$ of the weighted Hilbert space
$\mathcal H := L^2(\Omega\times\mathbb R^K;M_0^{-1}\,\d x\,\d\yy)$,
defined by
\[
   \mathcal H_{0}^{\mathcal R} :=  \left\{ h\in \mathcal H; \,  \mathcal Rh=h \mbox{ and }
        \int_{\Omega}\int_{\mathbb R^K} h(x,\yy)\,\d\yy\,\d x=0  \right\},
\]
with the reflection operator $\mathcal R h(x,\yy):=h(-x,-\yy)$.
The restriction to even functions $Rh=h$ is made in order to remove the degeneracy caused by translation invariance over the torus $\Omega$.
We then define the space $\mathcal X  := \left\{   h\in\mathcal H_0^{\mathcal R}: \mathcal L[\varrho_0]h-y^1\cdot\grad_x h\in\mathcal H  \right\}$,
equipped with the corresponding graph norm.
Note that since $\mathcal F$ is equivariant with respect to $\mathcal R$,
\[
    \mathcal F(\varrho,\mathcal Rh)=\mathcal R\mathcal F(\varrho,h),
\]
the restriction $\mathcal F:(0,\infty)\times\mathcal X \to   \mathcal H_{0}^{\mathcal R}$ is well-defined.
The $C^2$-smoothness of $\mathcal F$, as required by the Crandall--Rabinowitz theorem,
is inherited from the assumed $C^2$-regularity of $G$, together with the linearity of the maps
$h\mapsto\rho[h]$ and $\varrho\mapsto W\ast\varrho$.

%Since $W$ is smooth and $G\in C^3$, the map $\mathcal F$ is $C^2$ between the natural weighted hypoelliptic function spaces.
%Let us now fix $\varrho_0>0$.
The Fr\'echet derivative of $\mathcal F(\varrho_0,h)$ with respect to the second argument, evaluated at $h=0$,
is given by the linearization \eqref{eq:linearized-h}, i.e.,
\[
   D_h \mathcal F(\varrho_0,0) \varphi =  \mathcal{L}[\varrho_0] \varphi - y^1\cdot\grad_x \varphi + \varrho_0 g_0 g_1 (W\ast\rho[\varphi]) \laplace_{y^K} M[\varrho_0] \qquad\mbox{for } \varphi\in\mathcal X.
\]
%with $\int_\Omega \rho[h](x) \,\d x = 0$.
We decompose the above operator into Fourier modes, and following the same steps as in Section \ref{subsec:linearized-stability},
we conclude that the $n$-th Fourier mode of a nontrivial element of its kernel satisfies the dispersion relation \eqref{eq:disp0} with $\lambda=0$, i.e.,
\(  \label{cond:disp}
   - \varrho_0 g_0 g_1 \widehat W_{n} |q_n|^2 I_n^0(\alpha)(\varrho_0) = 1.
\)
Here we used the invertibility of $\mathcal B_n$, established in Remark \ref{rem:spectrum-Bn} of the Appendix.

By assumptions \eqref{ass:kyes}--\eqref{ass:nno}, condition \eqref{cond:disp} holds for exactly one positive Fourier mode $k\geq 1$.
The neutral zero mode corresponding to changes of total mass is excluded by the fixed-mass constraint.
%The zero Fourier mode is excluded by the vanishing total mass of $\varphi\in\mathcal X$.
Therefore, the kernel of $D_h \mathcal F(\varrho_0,0)$ in the even subspace $\mathcal X$ is one-dimensional and is spanned by
$\Re\left( e^{iq_k x} \Phi_k(\yy)\right)$, with $\Phi_k=\Phi_k(\yy)$ given by \eqref{def:Phik}.
Assumption \eqref{ass:kyes} implies the normalization $\rho[\Phi]=1$.
Moreover, by the Fredholm property of $\mathcal L$, obtained from the spectral analysis in Section \ref{sec:spectral} of the Appendix,
and stability of the Fredholm index under compact perturbation, $D_h \mathcal F(\varrho_0,0)$ is Fredholm of index zero on $\mathcal X$,
see, e.g., \cite{Kato}.

Then, if the Crandall--Rabinowitz \cite{Rabinowitz} transversality condition
\(   \label{cond:CR}
    \frac{\d}{\d\varrho} \left[ \varrho G(\varrho) G'(\varrho) \widehat W_{k} |q_k|^2 I_0^k(\varrho)   \right]_{\varrho=\varrho_0} \neq 0
\)
holds, there exists a local curve
\[
    s\mapsto (\varrho_s,h_s),    \qquad |s|<\delta,
\]
with $h_0=0$, such that $\mathcal F(\varrho_s,h_s)=0$, and
\[
    h_s(x,\yy) = s\Re\left( e^{iq_k x} \Phi_k(\yy)\right) + o(s),
\]
where $o(s)$ is taken with respect to the norm in $\mathcal X$.
Obviously, $f_s:= \varrho_s M[\varrho_s] + h_s$ is a family of spatially inhomogeneous 
stationary solutions of the nonlinear kinetic Fokker--Planck equation \eqref{eq:FPstat}.

It therefore remains to verify the transversality condition \eqref{cond:CR}.
It is convenient to define
\(   \label{def:J}
    J(\beta) := \beta \int_0^\infty F_K(\tau)^2 \exp\left( -\beta \widetilde V(\tau)  \right) \d\tau,
    \qquad
    \beta = \beta_k(\varrho) := \frac{|q_k|^2 G(\varrho)^2}{2\alpha^{2K+1}}.
\)
Then \eqref{cond:CR}, after division by $\widehat W_{k}|q_k|^2 \neq 0$, is then rewritten as
\[
    0 &\neq&  \tot{}{\varrho}\left[  \varrho \, (\ln G(\varrho))'  J(\beta_k(\varrho)) \right]_{\varrho=\varrho_0}  \\
      &=& [ \varrho \, (\ln G(\varrho))' ]' \, J(\beta_k(\varrho_0))  + \varrho_0\, (\ln G(\varrho))'  \, [J(\beta_k(\varrho))]', %\tot{}{\varrho} J(\beta_k(\varrho)),
\]
where all derivatives (marked by prime) are taken with respect to $\varrho$ and evaluated at $\varrho=\varrho_0$.
We readily have $\beta'_k(\varrho) = 2(\ln G(\varrho))' \beta_k(\varrho)$, so that we arrive at
\[
    0 \neq [ \varrho \, (\ln G(\varrho))' ]' J(\beta_k(\varrho_0))  + 2\varrho_0\, [(\ln G(\varrho))']^2 \beta_k(\varrho_0) J'(\beta_k(\varrho_0)).
\]
Since, by definition, $J(\beta)>0$ for all $\beta>0$, and, by assumption \eqref{ass:trans}, $[ \varrho \, (\ln G(\varrho))' ]' < 0$, the first term
of the right-hand side above is strictly negative. Consequently, it suffices to prove
that $J'(\beta) = \tot{}{\beta} J(\beta) < 0$ for all $\beta>0$, which makes the second term strictly negative as well.

From the proof of Lemma \ref{lem:I0monotone} we recall that $\widetilde V_K$ is strictly increasing
and maps $(0,\infty)$ to itself.
We then carry out the change of variables $u=\beta \widetilde V_K(\tau)$ in \eqref{def:J}, which gives
\[
    J(\beta) =    \int_0^\infty \frac{F_K(\tau)^2}{\widetilde V'(\tau)}  e^{-u}\,\d u
    \qquad\mbox{with } \tau = \widetilde V^{-1}(u/\beta).
\]
Moreover, in the proof of Lemma \ref{lem:I0monotone} we have shown that the function
$\tau \mapsto \frac{F_K(\tau)^2}{\widetilde V'(\tau)}$
is strictly increasing on $(0,\infty)$.
Then, since $\beta \mapsto \widetilde V^{-1}(u/\beta)$ is strictly decreasing in $\beta$ for each $u>0$,
we conclude that $J=J(\beta)$ is strictly decreasing on $(0,\infty)$.
\end{proof}

Let us finally consider the case when $W=W(x)$ is the normalized top-hat kernel $W= \frac{1}{2R}\chi_{[-R,R]}$ with $0<R<1/2$,
used for numerical simulations in \cite{BHW:2012} and \cite{EH:2026}.
Then
\(  \label{Wnhat}
    \widehat W_n = \frac{\sin(2\pi nR)}{2\pi n R}   \qquad \mbox{for }n\neq0.
\)
We prove that if the first Fourier mode is critical, then it is automatically simple.
Indeed, the criticality condition \eqref{ass:kyes} for $k=1$ reads
\[
    - 2 \varrho_0 \frac{g_1}{g_0} \widehat W_1 J(\beta_1)=1,
\]
with $J=J(\beta)$ and $\beta_1$ given by \eqref{def:J}.
Clearly, $\beta_1 < \beta_n$ for all $n\geq 2$, and since $J=J(\beta)$ is decreasing, we have $J(\beta_1) > J(\beta_n)$.

Moreover, Lemma \ref{lem:technical} gives, for any $\theta\in (0,\pi)$,
\[
    \sin(n\theta) \leq |\sin(n\theta)|<n\sin\theta \qquad\mbox{for all } n\geq 2.
\]
Using this with $\theta:=2\pi R\in(0,\pi)$ in \eqref{Wnhat}, we conclude that $\widehat W_1 > \widehat W_n$ for all $n\geq2$.
Therefore, $J(\beta_1) \widehat W_1 > J(\beta_n) \widehat W_n$ for all $n\geq 2$ such that $\widehat W_n > 0$,
and \eqref{ass:nno} holds.
Finally, if $\widehat W_n\leq 0$, then the left-hand side of
\eqref{ass:nno} is nonpositive, since $g_1<0$, and hence the mode cannot
be critical.

%%%%%%%%%%%%%%%%%%%%%%%%%%%%%%%%%%%%%%%%
\section{A visual example}\label{sec:visual}
The goal of this section is to demonstrate that for particular values of the model parameters
the dispersion relation \eqref{eq:general-K-dispersion-alpha} can be efficiently resolved numerically,
and information about the number of unstable Fourier modes
as a function of $K\in\N$ and $\alpha>0$ can be obtained.
With an example, we demonstrate that longer memory, measured in terms of $K/\alpha$
according to formula \eqref{eq:MemoryLength},
leads to fewer unstable Fourier modes, and, in effect, to less oscillatory aggregation patterns
(or, eventually, to no pattern formation at all).

Inspired by the extensive numerical simulations carried out in \cite{EH:2026},
we choose the response function $G(s)$ and the interaction kernel $W(x)$ as
\[
    G(s) := e^{-s}   \qquad \mbox{and} \qquad W(x) := \frac{\chi_{[0,R]}(|x|)}{2R},
\]
where $\chi_{[0,R]}$ denotes the characteristic function of the interval $[0,R]$,
{i.e.}, the kernel $W=W(x)$ corresponds to the normalized sampling radius $R>0$.
The Fourier coefficients of $W$ are given by \eqref{Wnhat}.
We choose $R:=0.05$ and unit total mass $\varrho_0:=1$.
Since $\frac{g_1}{g_0} = \frac{G'(\varrho_0)}{G(\varrho_0)} = -1$,
the condition for linear instability of the homogeneous steady state \eqref{eq:alphaCrit} reads
\(  \label{eq:visual1}
   \frac{\sin(2\pi nR)}{2\pi n R} > \frac12.
\)
If \eqref{eq:visual1} is verified, then for each $K\in\N$ there exists a unique critical value
$\bar\alpha=\bar\alpha(K,n)>0$, found as the solution of \eqref{eq:dispI0n},
such that the $n$-th Fourier mode is unstable for all $\alpha>\bar\alpha$
(i.e., for sufficiently short memory lengths).
It can be easily checked numerically that for $R=0.05$, \eqref{eq:visual1} holds iff $|n| \leq 6$.
We therefore focus on numerical search of the unique real root $\lambda\in\R$ of the dispersion
relation \eqref{eq:general-K-dispersion-alpha} for the first six nonzero Fourier modes.

We restrict ourselves to the cases $K=1,2,3$, where we derived the explicit dispersion
relations \eqref{eq:dispK1},  \eqref{eq:dispK2} and \eqref{eq:dispK3} in Sections \ref{subsec:K1}, \ref{subsec:K2} and \ref{subsec:K3}
of the Appendix.
We searched for their unique real roots $\lambda=\lambda_n(\alpha)\in\R$ numerically using the {\tt fzero} function of \textsc{Matlab},
with the code published in \cite{Matlab}. The results are plotted in Fig. \ref{fig:lambda}.
We also list the critical values $\bar\alpha=\bar\alpha(K,n)>0$ found as the unique solutions
of \eqref{eq:dispI0n} in Table \ref{tab:critAlpha}.
We observe that for each fixed $K$, the critical values are increasing with $n$.
Therefore, increasing the effective memory length
$K/\alpha$ by decreasing $\alpha$ reduces the number of unstable Fourier modes,
and, therefore, produces less oscillatory aggregation patterns in the linearized regime.
Ultimately, if $\alpha<\bar\alpha(K,1)$, all spatial Fourier modes are linearly stable, and the
linearized dynamics no longer predicts aggregation from small perturbations
of the homogeneous steady state.

\begin{figure}[ht]
{\centering
\resizebox*{0.27\linewidth}{!}{\includegraphics{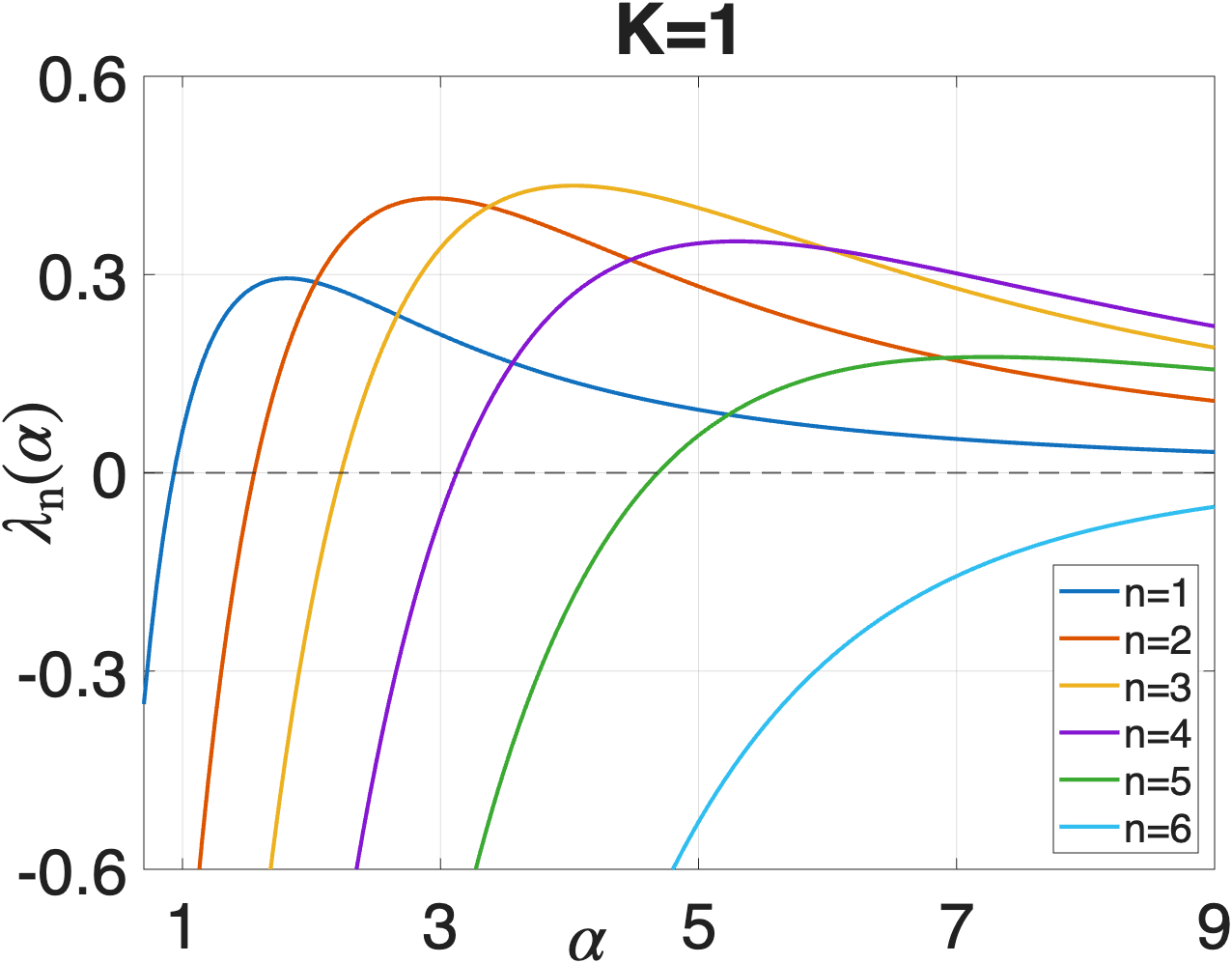}}
$\qquad$
\resizebox*{0.27\linewidth}{!}{\includegraphics{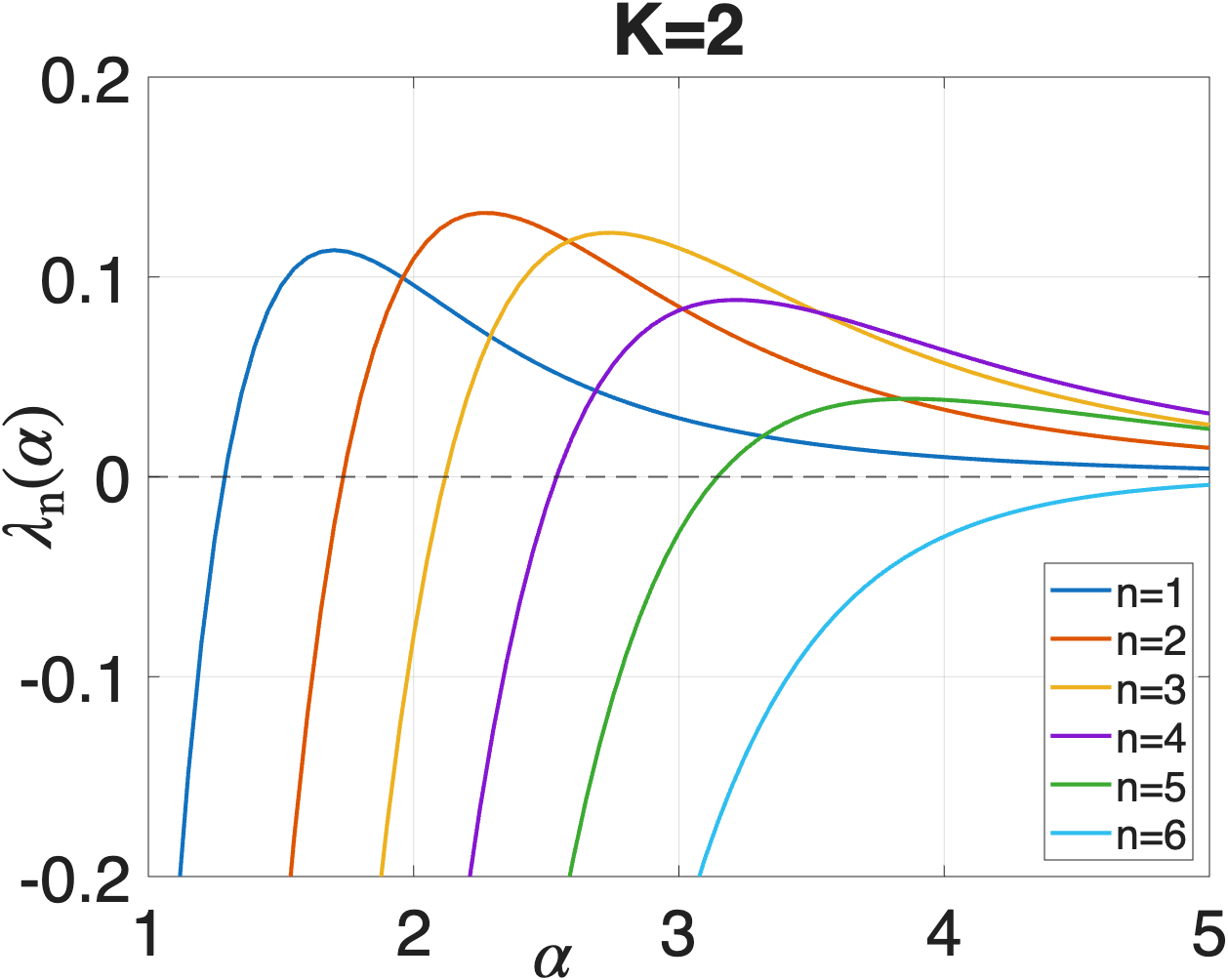}}
$\qquad$
\resizebox*{0.27\linewidth}{!}{\includegraphics{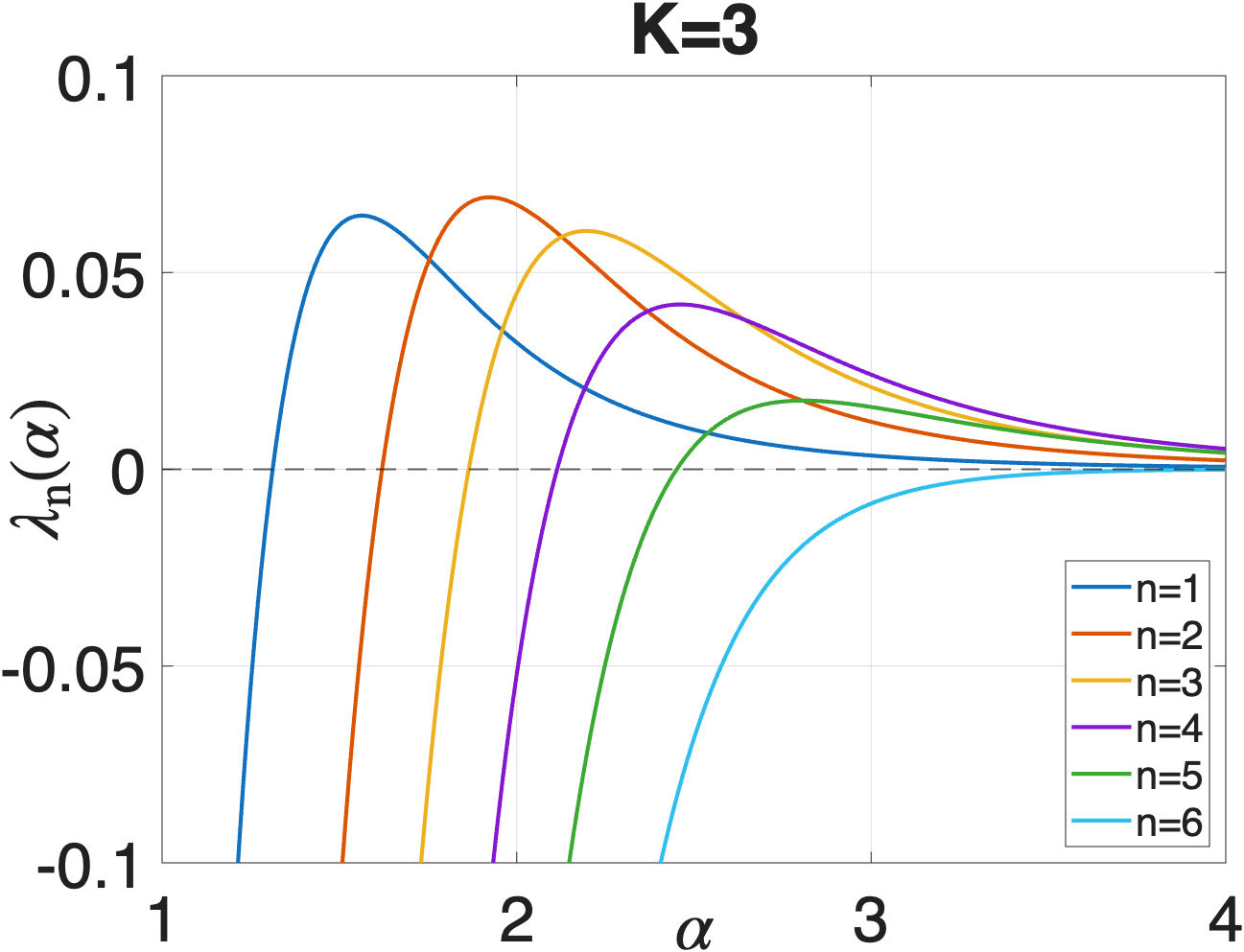}}
\par}
\caption{Unique real roots $\lambda=\lambda_n(\alpha)\in\R$ of the dispersion relation \eqref{eq:dispK1} for $K=1$
(left panel), \eqref{eq:dispK2} for $K=2$ (middle panel) and \eqref{eq:dispK3} for $K=3$ (right panel)
for the Fourier modes $n=1,2,\dots,6$.
\label{fig:lambda}}
\end{figure}

\begin{table}[ht]
\centering
\begin{tabular}{c c c c c c c}
\toprule
$K\backslash n$ & $1$ & $2$ & $3$ & $4$ & $5$ & $6$ \\
\midrule
$1$ & 0.932 & 1.555 & 2.229 & 3.125 & 4.684 & 17.381 \\
$2$ & 1.287 & 1.734 & 2.119 & 2.538 & 3.145 & 6.671 \\
$3$ & 1.312 & 1.621 & 1.865 & 2.115 & 2.452 & 4.168 \\
%$1$ & 0.9323 & 1.5552 & 2.2288 & 3.1247 & 4.6841 & 17.3811 \\
%$2$ & 1.2870 & 1.7335 & 2.1185 & 2.5384 & 3.1453 & 6.6709 \\
%$3$ & 1.3120 & 1.6210 & 1.8650 & 2.1150 & 2.4520 & 4.1680
%[0.9320    1.5550    2.2290    3.1250    4.6840   17.3810; 1.2870    1.7340    2.1190    2.5380    3.1450    6.6710; 1.3120    1.6210    1.8650    2.1150    2.4520    4.1680]
\bottomrule
\end{tabular}
\caption{Critical values of the relaxation parameter $\bar\alpha=\bar\alpha(K,n)>0$ found as the unique solutions
of \eqref{eq:dispI0n}, for $K=1$ (first row), $K=2$ (second row) and $K=3$ (third row),
and Fourier modes $n=1,\dots,6$.
The $n$-th Fourier mode is unstable for all $\alpha>\bar\alpha(K,n)$.}
\label{tab:critAlpha}
\end{table}

The comparison between different values of $K$ is more subtle, because
changing $K$ changes not only the effective memory length $K/\alpha$,
but also the shape of the memory kernel \eqref{eq:kappa},
cf. \cite[Figure 1]{EH:2026}.
To gain an insight into this phenomenon, we plotted the values of $K\bar\alpha(K,n)^{-1}$,
with $\bar\alpha(K,n)$ taken from Table \ref{tab:critAlpha}, against the Fourier mode numbers $n$.
The result is presented in Fig. \ref{fig:KoverBarAlpha}, where we clearly observe the monotone pattern.
The figure shall be interpreted as follows: The $n$-th Fourier mode
is unstable in the model with $K$ memory layers if $\frac{K}{\alpha} < \frac{K}{\bar\alpha(K,n)}$.
The plot tells us that increasing the memory length in terms of $K/\alpha$
leads to a reduction of the number of unstable Fourier modes, i.e., less oscillatory aggregation patterns,
or, ultimately, no patterns at all.

\begin{figure}[ht]
{\centering
\resizebox*{0.6\linewidth}{!}{\includegraphics{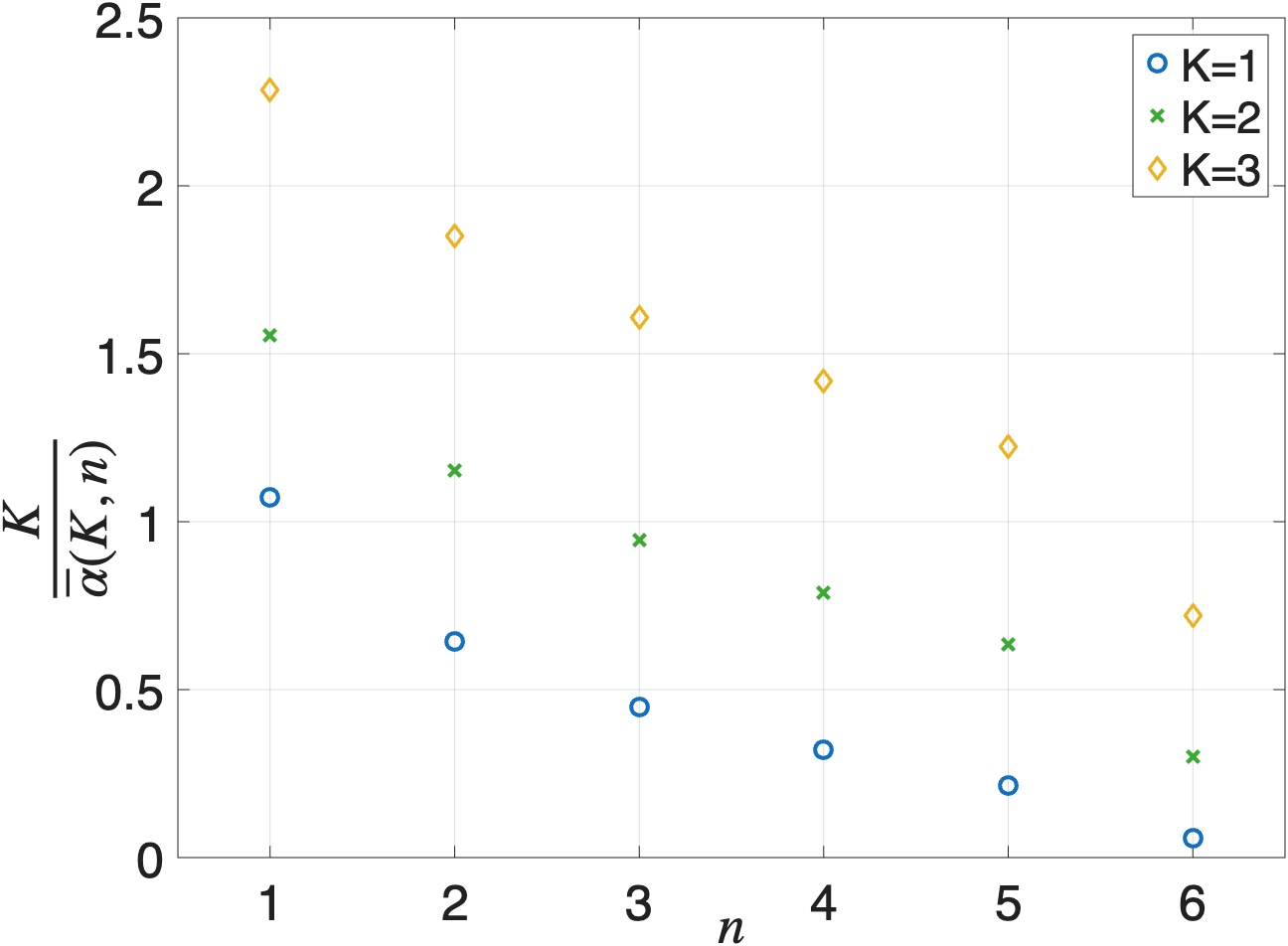}}
\par}
\caption{Values of the critical memory length $K{\bar\alpha(K,n)}^{-1}$ for $K=1$
(blue circles), $K=2$ (green crosses) and $K=3$ (orange diamonds),
for the Fourier modes $n=1,2,\dots,6$.
\label{fig:KoverBarAlpha}}
\end{figure}

%%%%%%%%%%%%%%%%%%%%%%%%%%%%%%%%%%%%%%%%
\section{Destabilization of a particular Fourier mode}\label{sec:destab}
We provide an example showing that the presence of memory
can destabilize a particular Fourier mode $n\neq 0$ that was stable
in the memoryless model.

We work in one spatial dimension and
choose $W(x) := \frac{1}{2R} {\chi_{[0,R]}(|x|)}$ together with
%the following parameters
\[
   K:=1,\qquad \varrho_0:=1,\qquad G(s):=e^{-25(s-1)},\qquad R:=0.05,\qquad n:=15.
\]
Then
\[
    g_0=1,\qquad g_1=-25,\qquad \widehat W_{15}=\frac{\sin(3\pi/2)}{3\pi/2} =-\frac{2}{3\pi}.
\]
We see that the mode $n=15$ is stable in the memoryless model because $\widehat W_{15}<0$
and so \eqref{eq:stability} is violated.
However, choosing $\alpha:=15$, numerical evaluation of \eqref{eq:dispK1} gives the roots
\[
   \lambda\approx 1.337 \pm 43.608\,i.
\]
Thus the mode is unstable in the memory model through an oscillatory instability.
But observe that the homogeneous state $\varrho_0=1$ is already unstable in the memoryless model, because
\[
   \widehat W_1=\frac{\sin(\pi/10)}{\pi/10}\approx0.98363>\frac1{50} = - \frac{g_0}{2g_1\varrho_0},
\]
and so \eqref{eq:stability} is verified for the first Fourier mode.
In fact, few more positive-Fourier modes are also unstable.

%%%%%%%%%%%%%%%%%%%%%%%%%%%%%%%%%%%%%%%%
\section{Appendix}\label{sec:App}

\subsection{Spectral analysis of $\mathcal B_n$}\label{sec:spectral}

Here we study the spectral properties of the operator $\mathcal B_n = \mathcal L_0 - i q_n\cdot \y^1$,
introduced for $n\neq 0$ in \eqref{def:Bn}, with $\mathcal L_0$ given by \eqref{def:L0}.
For this purpose we introduce the weighted Hilbert space $H:=L^2(M_0^{-1}\d\yy)$,
equipped with the inner product
\[
    \langle h_1,h_2\rangle_{H} := \int_{\mathbb R^{dK}} h_1(\yy)\overline{h_2(\yy)} \, M_0(\yy)^{-1} \,\d\yy,
\]
where $M_0$ is the invariant Gaussian defined in \eqref{eq:M0}.
The domain of $\mathcal B_n$ is the closure of $C_c^\infty(\mathbb R^{dK})$
with respect to the graph norm $\|h\|_{H}+\|\mathcal B_n h\|_{H}$.

We first prove a technical lemma.

\begin{lemma}  \label{lem:Aell-eK}
Let $K\geq 1$, $A\in\R^{K\times K}$ be given by \eqref{def:A}, and let $\ee_1,\dots, \ee_K$ denote the standard basis vectors of $\mathbb R^K$.
Then for all $\ell \in \{0, 1, \dots, K-1\}$ we have
\[
    A^\ell \ee_K = \ee_{K-\ell} + \mbox{l.c.}\{ \ee_{K-\ell+1},\dots,\ee_K \},
\]
where $\mbox{l.c.}$ stands for linear combination.
In particular, the vectors
\[
    \ee_K, A\ee_K, A^2\ee_K,\dots,A^{K-1}\ee_K
\]
are linearly independent.
%form a triangular family in the reversed basis $e_K,e_{K-1},\dots,e_1$, with all diagonal coefficients equal to $1$
\end{lemma}

\begin{proof}
The statement obviously holds for $\ell=0$.
We proceed inductively, assuming that it holds up to some $\ell\in [K-2]$.
A simple calculation reveals that
\[
    A \ee_j = \ee_{j-1} - \alpha_j \ee_j \qquad\mbox{for all } j\in[K],
\]
with the convention $\ee_0:=0$.
Consequently, we have
\[
   A \left( A^\ell \ee_K \right) &=& A \ee_{K-\ell} + A \left( \mbox{l.c.}\{ \ee_{K-\ell+1},\dots,\ee_K \} \right) \\
      &=& \ee_{K-\ell-1} - \alpha_{K-\ell} \ee_{K-\ell} + \mbox{l.c.}\{ \ee_{K-\ell},\dots,\ee_K \}.
\]
We conclude that
\[
   A^{\ell+1} \ee_K = \ee_{K-(\ell+1)} + \mbox{l.c.}\{ \ee_{K-\ell},\dots,\ee_K \}.
\]
\end{proof}

\begin{lemma}\label{lem:resolvent}
For any $n\neq 0$, $\mathcal B_n$ given by \eqref{def:Bn} generates a contraction semigroup on $\mathcal H$.
Consequently, the resolvent set of $\mathcal B_n$ contains the open right complex plane $\{\lambda\in\C;\, \Re\lambda >0 \}$
and the Laplace transform representation holds,
\(   \label{eq:resolvent}
    (\lambda-\mathcal B_n)^{-1} h = \int_0^\infty e^{-\lambda t}e^{t\mathcal B_n}h\,\d t \qquad \mbox{for all } \Re\lambda>0.
\)
\end{lemma}

\begin{proof}
We first prove that $\mathcal B_n$ is dissipative.
For smooth, rapidly decaying functions $h=M_0 f$ we have
\[
    \langle \mathcal L_0h,h\rangle_{H}
    = \int_{\R^{dK}} \mathcal L_0(M_0f)\,\overline f\,\d\yy 
    = \int_{\R^{dK}}M_0 f\,\mathcal L_0^\ast \overline f\,\d\yy ,
\]
where the Lebesgue adjoint $\mathcal L_0^\ast$ is the backward Ornstein--Uhlenbeck generator
\[
    \mathcal L_0^\ast f = (A\yy)\cdot\nabla_\yy f  +  \frac{g_0^2}{2}\Delta_{y^K}f.
\]
Using the identity
\[
    \mathcal L_0^\ast |f|^2 = 2\operatorname{Re}\left(f\, \mathcal L_0^\ast \overline f\right) + g_0^2|\nabla_{y^K}f|^2,
\]
we obtain
\[
\begin{aligned}
    \operatorname{Re}\langle \mathcal L_0 h,h \rangle_{H}
    &=
    \frac12\int_{\R^{dK}}M_0 \mathcal L_0^\ast |f|^2\,\d\yy
    -
    \frac{g_0^2}{2}\int_{\mathbb R^{dK}}|\nabla_{y^K}f|^2M_0\,\d\yy .
\end{aligned}
\]
Then, since $\mathcal L_0 M_0 = 0$, we have
\[
    \Re \langle \mathcal L_0 h,h \rangle_{H}
    =  -\frac{g_0^2}{2}  \int_{\R^{dK}}    \left|\nabla_{y^K}\left(\frac{h}{M_0}\right)\right|^2  M_0\,\d\yy.
\]
Finally, since
\[
    \left\langle -iq_n\cdot y^1h,h\right\rangle_{H}
    =   -i \int_{\mathbb R^{dK}}(q_n\cdot y^1)|h|^2M_0^{-1}\,\d\yy
\]
is purely imaginary, we arrive at
\(   \label{eq:diss}
    \Re \langle \mathcal B_n  h,h  \rangle_{H}
    =  -\frac{g_0^2}{2}  \int_{\R^{dK}} \left|\nabla_{y^K}\left(\frac{h}{M_0}\right)\right|^2   M_0\,\d\yy   \leq 0.
\)
By a standard density argument, we conclude that $\mathcal B_n$ is dissipative on its domain.

%The preceding identity proves dissipativity of $\mathcal B_n$ in $\mathcal H=L^2(M_0^{-1}\,d\yy)$. 
We now note that the drift matrix $A$ has all eigenvalues $-\alpha_k<0$ and,
by Lemma \ref{lem:Aell-eK}, satisfies the Kalman condition
\(  \label{eq:Kalman}
    \operatorname{span} \left\{\ee_K,A\ee_K,\dots,A^{K-1}\ee_K \right\} = \mathbb R^K.
\)
Therefore, the associated Ornstein--Uhlenbeck operator $\mathcal L_0^\ast$ generates a strongly continuous contraction semigroup in $L^2(M_0\,\d\yy)$, see, e.g., \cite{Metafune}.
Recalling the backward Fourier-mode operator \eqref{def:Cn}, we have %with the opposite sign
\[
    \mathcal C_{-n}   =    \mathcal L_0^\ast+iq_n\cdot y^1.
\]
By the Feynman--Kac formula \cite{Oksendal, Mao:2007}, its semigroup is given by
\[
       e^{t\mathcal C_{-n}} \varphi(\yy) = \mathbb E \left[ \exp\left(iq_n\cdot\int_0^t Y_s^1\,ds\right) \varphi(Y_t) \Bigg| \, Y_0= \yy \right],
\]
where $(Y_s)_{s\geq0}$ is the Ornstein--Uhlenbeck process generated by $\mathcal L_0^\ast$ and started from $Y_0=\yy$.
Then, by Jensen's inequality,
\[
    \left|  e^{t\mathcal C_{-n}} \varphi(\yy)   \right|^2    \leq    \mathbb E \left[ |\varphi(Y_t)|^2 \Bigg| \, Y_0= \yy \right],
\]
and integrating with respect to the invariant measure $M_0\,\d\yy$ gives
\[
    \Norm{e^{t\mathcal C_{-n}} \varphi}_{L^2(M_0)}^2  \leq \int_{\mathbb R^K}|\varphi(\yy)|^2M_0(\yy)\,d\yy = \Norm{\varphi}_{L^2(M_0)}^2.
\]
Thus $\mathcal C_{-n}$ generates a contraction semigroup on $L^2(M_0\,d\yy)$.

Finally, define the isometry
\[
    \mathcal{U}: H\to L^2(M_0\,d\yy),    \qquad    \mathcal{U} h:=\frac{h}{M_0}.
\]
Then $\mathcal{U}\mathcal B_n \mathcal{U}^{-1}$ %\[ \widetilde{\mathcal B}_n :=   U\mathcal B_nU^{-1}  \]
is the $L^2(M_0\,\d\yy)$-adjoint of $\mathcal C_{-n}$,
%Indeed, for smooth functions $u,v$,
%\[  \left\langle \widetilde{\mathcal B}_n u,v\right\rangle_{L^2(M_0)} =   \left\langle u,\mathcal C_{-n}v\right\rangle_{L^2(M_0)} . \]
and thus generates a contraction semigroup on $L^2(M_0\,d\yy)$.
By the isometry $\mathcal{U}$, we  conclude that the operator $\mathcal B_n$ generates a contraction semigroup on $\mathcal H$.
%Equivalently, the closure of $\mathcal B_n$ is maximal dissipative, and, therefore, $\lambda-\mathcal B_n$ is surjective for $\Re \lambda>0$.
%By the Lumer--Phillips theorem \cite{Lumer}, $\mathcal B_n$ generates a contraction semigroup on $\mathcal H$.
\end{proof}

\begin{remark}\label{rem:spectrum-Bn}
With standard arguments, one can strengthen the statement of Lemma \ref{lem:resolvent}.
Namely, using the hypoelliptic Ornstein--Uhlenbeck structure of $\mathcal B_n$, satisfying the Kalman condition \eqref{eq:Kalman},
and the confinement induced by the Gaussian invariant measure $M_0$, one obtains that the resolvent \eqref{eq:resolvent} is compact;
see, e.g., \cite{Metafune, Ottobre}.
Consequently, the spectrum of $\mathcal B_n$ consists only of isolated eigenvalues of finite algebraic multiplicity, with possible accumulation only at infinity.

Moreover, for $n\neq0$, one can exclude eigenvalues on the imaginary axis. Indeed, if $\mathcal B_n h=i\omega h$,
then the dissipativity \eqref{eq:diss} gives
\[
    \nabla_{y^K}\left(\frac{h}{M_0}\right)=0.
\]
Using the commutator structure of the Ornstein--Uhlenbeck chain and \eqref{eq:Kalman},
this degeneracy propagates through the variables $y^{K-1},\dots,y^1$, implying that $h/M_0$ is constant. But then the equation
$\mathcal B_n h=i\omega h$ would require
\[
    -iq_n\cdot y^1 h=i\omega h,
\]
which is impossible for $n\neq0$ unless $h=0$. Hence $\mathcal B_n$ has no eigenvalues on $i\mathbb R$.

Since the resolvent operator is compact, the absence of eigenvalues on $i\mathbb R$ implies the absence of spectrum on $i\mathbb R$.
Therefore, the statement of Lemma \ref{lem:resolvent} can be extended to assert that the resolvent set of $\mathcal B_n$, $n\neq 0$,
contains the closed right plane $\{\lambda\in\C; \Re\lambda\geq 0\}$.
%Thus the free kinetic operator has no spectrum in the closed right half-plane for any nonzero spatial Fourier mode.
\end{remark}

%%%%%%%%%%%%%%%%%%
\subsection{A technical lemma}

\begin{lemma}\label{lem:technical}
For any $\theta\in [0,\pi]$ and $n\in\Z$ with $|n|\geq 2$ we have
\[
      |\sin(n\theta)| \leq |n|\sin\theta,
\]
and the inequality is strict if $\theta\in (0,\pi)$.
\end{lemma}

\begin{proof}
Let $n\geq 2$. For $\theta=0$ and $\theta=\pi$ the claim is trivial.
We therefore take $\theta\in (0,\pi)$ and use the fact that $\frac{\sin(n\theta)}{\sin\theta} = U_{n-1}(\cos\theta)$,
where $U_{n-1}$ is the Chebyshev polynomial of the second kind.
The exponential representation formula \cite{Mason:2002} gives, for $n\geq2$,
\[
    \frac{\sin(n\theta)}{\sin\theta}  =  e^{i(n-1)\theta}+e^{i(n-3)\theta}+\cdots+e^{-i(n-1)\theta}.
\]
Since the right-hand side is a sum of $n$ complex numbers of modulus one, 
the triangle inequality gives $\left|\frac{\sin(n\theta)}{\sin\theta}\right|\leq n$.
For $n\leq -2$ the claim follows by symmetry.

The inequality is strict for $\theta\in(0,\pi)$, since equality in the
triangle inequality would require all phases to coincide, which would imply
$e^{2i\theta}=1$, i.e., $\theta=0$ or $\theta=\pi$. Hence
\[
    \sin(n\theta) \leq |\sin(n\theta)|<n\sin\theta \qquad\mbox{for all } n\geq 2.
\]
\end{proof}

%%%%%%%%%%%%%%%%%%
\subsection{Explicit dispersion relation for $K=1$}\label{subsec:K1}
Formula \eqref{eq:Cj} for $K=1$ reads
\begin{equation}
\label{eq:b1-def}
    C_1(t)  =  \int_0^t e^{-\alpha s}\,\d s  =  \frac{1-e^{-\alpha t}}{\alpha}.
\end{equation}
Moreover, if $Y_t$ is the stationary Ornstein--Uhlenbeck process
$\d Y_t = -\alpha Y_t\,\d t + g_0\,\d B_t$,
then, for one spatial component,
\[
    \operatorname{Cov}_{M_0}(Y_s,Y_u)  =  \frac{g_0^2}{2\alpha}e^{-\alpha|s-u|}.
\]
Consequently,
\(   \label{eq:Cov}
    V(t)  =  \operatorname{Var}_{M_0}   \left(    \int_0^t Y_s\,\d s  \right) 
     &=&   \int_0^t\int_0^t    \operatorname{Cov}_{M_0}(Y_s,Y_u)\, \d u\,\d s   \\
    %&=& \frac{g_0^2}{2\alpha}  \int_0^t\int_0^t   e^{-\alpha|s-u|}\,\d u\,\d s \\
    &=& \frac{g_0^2}{\alpha^2} t   -  \frac{g_0^2}{\alpha^3}  \left(1-e^{-\alpha t}\right).
    \nonumber
\)
Substituting into \eqref{eq:general-K-dispersion-alpha}, the dispersion relation becomes
\begin{equation}
\label{eq:dispK1}
    1 = -\varrho_0 g_0 g_1\,\widehat W_n |q_n|^2  \int_0^\infty  e^{-\lambda t}  \left( \frac{1-e^{-\alpha t}}{\alpha}  \right)^2
    \exp\left[  -\frac{|q_n|^2 g_0^2}{2\alpha^3}  \left( \alpha t - 1 + e^{-\alpha t} \right)  \right] \,\d t.
\end{equation}

%%%%%%%%%%%%%%%%%%
\subsection{Explicit dispersion relation for $K=2$}\label{subsec:K2}
For $K=2$ and $\alpha_1=\alpha_2=\alpha>0$,
%the internal Ornstein--Uhlenbeck chain is
%\[  dY_s^1=(Y_s^2-\alpha Y_s^1)\,ds,\qquad dY_s^2=-\alpha Y_s^2\,ds+g_0\,dB_s . \]
we have
%\[   A=\begin{pmatrix} -\alpha & 1\\ 0 & -\alpha \end{pmatrix}, \]
\[
    e^{sA}=e^{-\alpha s}\begin{pmatrix} 1 & s\\ 0 & 1 \end{pmatrix},\qquad
    \Sigma=g_0^2\begin{pmatrix} \frac{1}{4\alpha^3} & \frac{1}{4\alpha^2}\\[2mm] \frac{1}{4\alpha^2} & \frac{1}{2\alpha} \end{pmatrix}.
\]
Consequently, \eqref{eq:Cj} reads
\[
    C_2(t)=\int_0^t(e^{sA})_{12}\,\d s=\int_0^t e^{-\alpha s}s\,\d s=\frac{1-e^{-\alpha t}(1+\alpha t)}{\alpha^2},
\]
and for $s$,  $u \geq 0$ we have, by stationarity and symmetry,
\[
    \operatorname{Cov}_{M_0}(Y_s^1,Y_u^1)
    =\left(e^{|s-u|A}\Sigma\right)_{11}
     %=e^{-\alpha(s-u)}\left(\Sigma_{11}+(s-u)\Sigma_{21}\right) \\
    =\frac{g_0^2}{4\alpha^3}e^{-\alpha|s-u|}\left(1+\alpha|s-u|\right).
\]
Inserting this into the formula \eqref{eq:Cov},
%\[  V(t) = \operatorname{Var}_{M_0}\left(\int_0^tY_s^1\,\d s\right) = \int_0^t\int_0^t\operatorname{Cov}_{M_0}(Y_s^1,Y_u^1)\,\d u\,\d s, \]
a simple calculation yields
\[
    V(t)
       %&=2\int_0^t\int_0^s \frac{g_0^2}{4\alpha^3}e^{-\alpha(s-u)}\left(1+\alpha(s-u)\right)\,du\,ds \\
       %&=\frac{g_0^2}{2\alpha^3}\int_0^t(t-\tau)e^{-\alpha\tau}(1+\alpha\tau)\,d\tau, \\
    =\frac{g_0^2}{2\alpha^5}\left(  2\alpha t-3+(\alpha t+3)e^{-\alpha t}\right).
\]
Substituting into \eqref{eq:general-K-dispersion-alpha}, we arrive at
\begin{equation}
\label{eq:dispK2}
\begin{aligned}
    1  &=
    -\varrho_0 g_0 g_1\widehat W_n |q_n|^2 \int_0^\infty e^{-\lambda t}\left(\frac{1-e^{-\alpha t}(1+\alpha t)}{\alpha^2}\right)^2  \\
    &\qquad\qquad\qquad\times
    \exp\left[ -\frac{|q_n|^2g_0^2}{4\alpha^5}\left(2\alpha t-3+(\alpha t+3)e^{-\alpha t}\right)\right] \d t.
\end{aligned}
\end{equation}

%%%%%%%%%%%%%%%%%%%%%%%%%%%%%%%%%%%%%
\subsection{Explicit dispersion relation for $K=3$}\label{subsec:K3}
For $K=3$ and $\alpha_1=\alpha_2=\alpha_3=\alpha>0$, we have
\[
    e^{sA}
    =
    e^{-\alpha s}
    \begin{pmatrix}
        1 & s & \frac{s^2}{2} \\
        0 & 1 & s \\
        0 & 0 & 1
    \end{pmatrix}, \qquad
    \Sigma
    =
    g_0^2
    \begin{pmatrix}
        \frac{3}{16\alpha^5} & \frac{3}{16\alpha^4} & \frac{1}{8\alpha^3} \\[2mm]
        \frac{3}{16\alpha^4} & \frac{1}{4\alpha^3} & \frac{1}{4\alpha^2} \\[2mm]
        \frac{1}{8\alpha^3} & \frac{1}{4\alpha^2} & \frac{1}{2\alpha}
    \end{pmatrix}.
\]
Consequently, \eqref{eq:Cj} reads
\[
    C_3(t) = \int_0^t(e^{sA})_{13}\,\d s
    = \frac12\int_0^t e^{-\alpha s}s^2\,\d s
    = \frac{1-e^{-\alpha t}\left(1+\alpha t+\frac{(\alpha t)^2}{2}\right)}{\alpha^3},
\]
and for $s$, $u\geq 0$ we have, by stationarity and symmetry,
\[
\begin{aligned}
    \operatorname{Cov}_{M_0}(Y_s^1,Y_u^1)
    &=
    \left(e^{|s-u|A}\Sigma\right)_{11} \\
    &=
    \frac{g_0^2}{16\alpha^5}
    e^{-\alpha|s-u|}
    \left(3+3\alpha|s-u|+\alpha^2|s-u|^2\right).
\end{aligned}
\]
Inserting this into the formula \eqref{eq:Cov},
%\[  V(t) = \operatorname{Var}_{M_0}\left(\int_0^tY_s^1\,\d s\right) = \int_0^t\int_0^t\operatorname{Cov}_{M_0}(Y_s^1,Y_u^1)\,\d u\,\d s, \]
we get
\[
    V(t)
    % = \frac{g_0^2}{8\alpha^5}  \int_0^t (t-\tau)e^{-\alpha\tau}  \left(3+3\alpha\tau+\alpha^2\tau^2\right)\,\d\tau 
    = \frac{g_0^2}{8\alpha^7} \left[ 8\alpha t-15 +  e^{-\alpha t}  \left((\alpha t)^2+7\alpha t+15\right)  \right].
\]
Substituting into \eqref{eq:general-K-dispersion-alpha}, we arrive at
\begin{equation}
\label{eq:dispK3}
\begin{aligned}
    1
    &=
    -\varrho_0 g_0 g_1\widehat W_n |q_n|^2
    \int_0^\infty
    e^{-\lambda t}
    \left(
        \frac{
            1-e^{-\alpha t}
            \left(1+\alpha t+\frac{(\alpha t)^2}{2}\right)
        }{\alpha^3}
    \right)^2  \\
    &\qquad\qquad\qquad\times
    \exp\left[
        -\frac{|q_n|^2g_0^2}{16\alpha^7}
        \left(
            8\alpha t-15
            +
            e^{-\alpha t}
            \left((\alpha t)^2+7\alpha t+15\right)
        \right)
    \right]\,\d t .
\end{aligned}
\end{equation}
\vspace{5mm}

\section*{Declarations}

\subsection*{Conflict of interest}
\noindent
The author declares no conflict of interests.

\subsection*{Data \& code availability}
\noindent
No datasets were generated or analyzed in this study. The \textsc{Matlab} code used to generate the numerical results
in Section \ref{sec:visual} is publicly available at \href{https://doi.org/10.5281/zenodo.21864715}{doi.org/10.5281/zenodo.21864715}.

\subsection*{Ethics}
This study did not involve human participants, animals, or sensitive data requiring ethical approval
or consent to participate.

\subsection*{Funding information}
This study was conducted using internal institutional resources. No external funding was received.
\vspace{5mm}

%%%%%%%%%%%%%%%%%%%%%%%%%%%%%%%%%%

\end{document}